\documentclass[10pt]{amsart}

\usepackage{courier}
\usepackage{amsmath,amsthm,amssymb,mathrsfs,amscd}
\usepackage[cmtip,all]{xy}
\usepackage{exscale}

\newtheorem{theorem}{Theorem}[section]
\newtheorem*{theorem*}{Theorem}
\newtheorem{lemma}[theorem]{Lemma}
\newtheorem*{lemma*}{Lemma}
\newtheorem{proposition}[theorem]{Proposition}
\newtheorem*{proposition*}{Proposition}

\newtheorem{corollary}[theorem]{Corollary}

\theoremstyle{definition}
\newtheorem{definition}[theorem]{Definition}
\newtheorem{example}[theorem]{Example}

\theoremstyle{remark}
\newtheorem{remark}[theorem]{Remark}

\numberwithin{equation}{section}

\newcommand{\set}{\mathbf{Set}}
\newcommand{\tops}{\mathbf{Top}}

\newcommand{\ktop}{\mathbf{kTop}}

\newcommand{\fktop}{\mathrm{k}}
\newcommand{\tot}{\mathrm{Tot}}

\DeclareMathOperator{\colim}{colim}
\DeclareMathOperator{\res}{Res}

\begin{document}
\title{A Spectral Sequence Connecting Continuous With Locally Continuous Group Cohomology}
\author[M. Fuchssteiner]{Martin Fuchssteiner}
\email{martin@fuchssteiner.net}
\date{May 01, 2011}

\begin{abstract}
We present a spectral sequence connecting the continuous and 
'locally continuous' group cohomologies for topological groups. 
As an application it is shown that for contractible 
topological groups these cohomology concepts coincide. 
\end{abstract}

\maketitle

\section*{Introduction}

There exist various cohomology concepts for topological groups $G$ and
topological coefficient groups $V$ which take the topologies of the group and 
that of the coefficients into account. 
One is obtained by restricting oneself to the complex $C_c^* (G;V)$
continuous group cochains only whose cohomology is called the 
\emph{continuous group cohomology} $H_c (G;V)$. 
For abstract groups $G$ and $G$-modules $V$ the first cohomology group
$H^1 (G;V)$ classifies crossed morphisms modulo principal derivations, 
the second cohomology group $H^2 (G;V)$ classifies equivalence classes of
group extensions $V \hookrightarrow \hat{G} \twoheadrightarrow G$ 
and the third cohomology group $H^3 (G;V)$ classifies equivalence classes 
crossed modules with kernel $V$ and cokernel $G$ 
(cf. \cite[Theorem 6.4.5, Theorem 6.6.3 and Theorem 6.6.13]{WeHA}. 
Analogous considerations show that for topological groups $G$ and $G$-modules 
$V$ the first cohomology group $H_c^1 (G;V)$ classifies continuous 
crossed morphisms modulo principal derivations, 
the second cohomology group $H_{c}^2 (G;V)$ classifies equivalence classes
of topological group extensions 
$V \hookrightarrow \hat{G} \twoheadrightarrow G$ which admit a global section
(i.e. $\hat{G} \twoheadrightarrow G$ is a  trivial $V$-principal bundle) 
and the third cohomology group $H_{c}^3 (G;V)$ classifies equivalence
classes of topologically split crossed modules. The continuous group
cohomology has the drawback that for even the compact Hausdorff group 
$G=\mathbb{R}/\mathbb{Z}$ the short exact sequence 
\begin{equation*}
 0 \rightarrow \mathbb{Z} \hookrightarrow \mathbb{R}
  \twoheadrightarrow \mathbb{R}/\mathbb{Z} \rightarrow 0
\end{equation*}
of coefficients does not induce a long exact sequence of cohomology groups. 
(The group $H_{c}^1 (G;\mathbb{R})$ is trivial because the projection 
$\mathbb{R} \twoheadrightarrow \mathbb{R}/\mathbb{Z}$ does not admit global 
sections, 
$H_{c}^n (G;\mathbb{Z})=0$ because all continuous group cochains on $G$ are 
constant whereas the group of 
$H_{c}^1 (G;G)$ all continuous endomorphisms of $G$ is non-trivial.)

This drawback is relieved by a second more general cohomology concept, which 
is obtained by considering the complex $C_{cg}^* (G;V)$ of group 
cochains which are continuous on some identity neighbourhood in $G$. 
By abuse of language some people call the
corresponding cohomology groups $H_{cg} (G;V)$ the 
`locally continuous group cohomology´. 
The first cohomology group $H_{cg}^1 (G;V)$ classifies continuous 
crossed morphisms modulo principal derivations, 
the second cohomology group $H_{cg}^2 (G;V)$ classifies equivalence classes
of topological group extensions 
$V \hookrightarrow \hat{G} \twoheadrightarrow G$ which admit local sections 
(i.e. $\hat{G} \twoheadrightarrow G$ is a locally trivial $V$-principal bundle) 
and the third cohomology group $H_{cg}^3 (G;V)$ classifies equivalence
classes of crossed modules in which all homomorphisms admit local sections.

The inclusion $C_{cg}^* (G;V) \hookrightarrow C_c^* (G;V)$ of cochain
complexes induces a morphism $H_{cg}^* (G;V) \rightarrow H_c^* (G;V)$ of 
cohomology groups, which is used to compare the two cohomology concepts.
As the above example shows, these cohomology concepts do not even coincide for 
connected compact Hausdorff groups and real coefficients. In the following we
will show that the contractibility of a topological group $G$ forces the two
cohomologies to coincide (cf. Corollary \ref{gcontriso}):

\begin{theorem*}
For contractible groups $G$ the inclusion 
$C_c^* (G;V) \hookrightarrow C_{cg}^* (X;V)$ induces an isomorphism in
cohomology. 
\end{theorem*}

This is proved by constructing a row-exact double complex 
$A_{cg,eq} (G;V)$ whose rows and columns can be augmented by the complexes 
$C_{cg}^* (G;V)$ and $C_{c}^* (G;V)$ respectively. The contractibility of $G$
will be shown to force the columns of this double complex to be exact as 
well, which then in turn is shown to imply that the inclusion  
$C_{cg}^* (G;V) \hookrightarrow C_c^* (G;V)$ induces an isomorphism in
cohomology. 
In fact we will be considering the more general setting of transformation
groups $(G,X)$ and $G$-equivariant cochains on $X$ and prove these results in
this more general setting. Similar results for $k$-groups and smooth 
transformation groups will also be obtained. 

\section{Basic Concepts}
In this section we recall the definitions of various cochain complexes and
the interpretation of some of their cohomology groups. 
For topological spaces $X$ and abelian topological groups $V$ one can consider 
variations of the exact \emph{standard complex} $A^* (X;V)=\hom_\set (X;V)$ of 
abelian groups. 

\begin{definition}
For every topological space $X$ and abelian topological group $V$ the 
subcomplex $A_c^* (X;V):= C (X^{*+1};V)$ of the standard complex is called the 
\emph{continuous standard complex}. 
\end{definition}

For transformation groups $(G,X)$ and $G$-modules $V$ the group $G$ acts on
the spaces $X^{n+1}$ via the diagonal action and the groups $A^n (X;V)$ can be 
endowed with a $G$-action via  
\begin{equation}\label{defgact}
  G \times A^n (X;V) \rightarrow A^n (X;V), \quad 
[g.f] (\vec{x})=g.[ f (g^{-1} .\vec{x})] \, .
\end{equation}
The $G$-fixed points of this action are the $G$-equivariant cochains. Because 
the differential of the standard complex intertwines the $G$-action, the
equivariant cochains form a subcomplex $A^* (X;V)^G$ of the standard complex
and the continuous equivariant cochains form a subcomplex $A_c^* (X;V)^G$ of
the continuous standard complex. These complexes not exact in general.

\begin{example}
  For any group $G$ which acts on itself by left translation and $G$-module
  $V$ the complex $A^* (G;V)^G$ is the complex of (homogeneous) group cochains; 
for topological groups $G$ and $G$-modules $V$ the complex 
$A_c^* (G;V)^G$ is the complex of continuous (homogeneous) group cochains
\end{example}

\begin{definition}
  The cohomology $H_{eq} (X;V)$ of the complex $A^* (X;V)^G$ is called the
  equivariant cohomology of $X$ (with values in $V$).
The cohomology $H_{eq,c} (X;V)$ of the subcomplex $A_c^* (X;V)^G$ is called the
  equivariant continuous cohomology of $X$ (with values in $V$).  
\end{definition}

%

\begin{example}
  For any group $G$ which acts on itself by left translation and $G$-module
  $V$ the cohomology $H_{eq} (G;V)$ is the group cohomology of $G$ with values 
in $V$; for topological groups $G$ and $G$-modules $V$ the cohomology 
$H_{eq,c} (G;V)$ is the continuous group cohomology of $G$ with values in $V$.  
\end{example}

For transformation groups $(G,X)$ and $G$-modules $V$ there exists a
$G$-invariant complex $A_{cg}^* (X;V)$ between $A_c (X;V)$ and $A (X;V)$ which 
we are going to define now. For each open covering $\mathfrak{U}$ of $X$ and 
each $n \in \mathbb{N}$ one can define an open neighbourhood 
$\mathfrak{U} [n]$ of the diagonal in $X^{n+1}$ via
\begin{equation*}
 \mathfrak{U} [n] := \bigcup_{U \in \mathfrak{U}} U^{n+1} \, .
\end{equation*}
These neighbourhoods of the diagonals in $X^{*+1}$ form a simplicial subspace 
of $X^{*+1}$ which allows us to consider the subcomplex of $A^* (X;V)$ formed
by the groups  
\begin{equation*}
A_{cr}^n (X,\mathfrak{U};V) := 
\left\{ f \in  A^n (X;V) \mid \, f_{\mid \mathfrak{U} [n]}
  \in C (\mathfrak{U}[n];V) \right\}  
\end{equation*}
of cochains whose restriction to the subspaces $\mathfrak{U}[n]$ of $X^{n+1}$ 
are continuous. The cohomology of the cochain complex 
$A_{cr}^* (X,\mathfrak{U};V)$ is denoted by $H_{cr} (X,\mathfrak{U};V)$. 
If the covering $\mathfrak{U}$ of $X$ is $G$-invariant, then the subspaces 
$\mathfrak{U}[*]$ is a simplicial $G$-subspace of the simplicial $G$-space 
$X^{*+1}$. 

\begin{example}
  If $G=X$ is a topological group which acts on itself by left translation and
  $U$ an open identity neighbourhood, then 
$\mathfrak{U}_U :=\{ g.U \mid g \in G \}$ is a $G$-invariant open covering of
$G$ and $\mathfrak{U}[*]$ is an open simplicial $G$-subspace of $G^{*+1}$.
\end{example}

For $G$-invariant coverings $\mathfrak{U}$ of $X$ the cohomology of the 
subcomplex $A_{cr}^* (X,\mathfrak{U};V)^G$ of $G$-equivariant cochains is 
denoted by $H_{cr,eq} (X,\mathfrak{U};V)$. 

\begin{example}
  If $G=X$ is a topological group which acts on itself by left translation and
  $U$ an open identity neighbourhood, then the complex 
$A_{cr}^* (X,\mathfrak{U}_U;V)^G$ is the complex of homogeneous group cochains 
whose restrictions to the subspaces $\mathfrak{U}_U [*]$ are continuous. 
(These are sometimes called $\mathfrak{U}$-continuous cochains.) 
\end{example}

For directed systems $\{ \mathfrak{U}_i \mid i \in I \}$ of open coverings of
$X$ one can also consider the colimit complex 
$\colim_i A_{cr}^* (X,\mathfrak{U}_i ;V)$. In particular for the directed 
system of all open coverings of $X$ one observes that the open diagonal 
neighbourhoods $\mathfrak{U}[n]$ in $X^{n+1}$ for open coverings 
$\mathfrak{U}$ of $X$ are cofinal in the directed set of all open diagonal 
neighbourhoods, hence one obtains the complex
\begin{equation*}
  A_{cg}^* (X;V):= \colim_{\mathfrak{U} \text{is open cover of $X$}} A_{cr}^*
  (X;\mathfrak{U};V)
\end{equation*}
of global cochains whose germs at the diagonal are continuous. This is a
subcomplex of the standard complex $A^* (X;V)$ which is invariant under the 
$G$-action (Eq. \ref{defgact}) and thus a sub complex of $G$-modules.  
The $G$-equivariant cochains with continuous germ form a subcomplex 
$A_{cg}^* (X;V)^G$ thereof, whose cohomology is denoted by $H_{cg,eq} (X;V)$. 
The latter subcomplex can also be obtained by taking the colimit over all 
$G$-invariant open coverings of $X$ only: 

\begin{proposition} \label{natinclofeqccisiso}
  The natural morphism of cochain complexes 
  \begin{equation*}
A_{cg,eq}^* (X;V):= \colim_{\mathfrak{U} \text{is $G$-invariant open cover
      of $X$}} A_{cr}^* (X;\mathfrak{U};V)^G \rightarrow A_{cg}^* (X;V)^G    
  \end{equation*}
is a natural isomorphism.
\end{proposition}

\begin{proof}
  We show that this morphism is surjective and injective. 
Every equivalence class in $A_{cg}^n (X;V)^G$ can be represented by a
cochain $f  \in A_{cr}^n (X,\mathfrak{U};V)^G$, where $\mathfrak{U}$ is an
open cover of $X$. The cochain $f$ is continuous on $\mathfrak{U}[n]$ by 
definition. Its equivariance implies, that it also is continuous on 
$G . \mathfrak{U}[n]=( G. \mathfrak{U})[n]$, 
hence an element of $A_{eq}^n (X, (G. \mathfrak{U})[n];V)$. The
equivalence class $[f] \in A_{cg,eq}^n (X;V)$ is mapped onto 
$[f]A_{cg}^n (X;V)^G$. This proves surjectivity.

Every equivalence class in $A_{cg,eq}^n (X;V)^G$ can be represented 
by an equivariant $n$-cochain $f$ in $A_{cr}^n (X, \mathfrak{U};V)^G$, where 
$\mathfrak{U}$ is a $G$-invariant open cover of $X$. If the image of the class
$[f] \in A_{cg}^* (X;V)^G$ is trivial, then the cochain $f$ itself is trivial
and so is its class $[f] \in A_{cg,eq}^n (X;V)^G$. This proves injectivity.
\end{proof}

\begin{corollary}
  The cohomology $H_{cg,eq} (X;V)$ is the cohomology of the complex of 
equivariant cochains which are continuous on some $G$-invariant neighbourhood
of the diagonal.
\end{corollary}

\begin{example}
  If $G=X$ is a topological group which acts on itself by left translation, 
then the complex $A_{cg}^* (G;V)^G$ is the complex of homogeneous group
cochains whose germs at the diagonal are continuous. 
(By abuse of language these are sometimes called 'locally continuous' 
group cochains.) 
\end{example}

\section{The Spectral Sequence}
\label{sectss}

Let $(G,X)$ be a transformation group, $V$ be $G$-module and $\mathfrak{U}$ be
an open covering of $X$. We will show (in Section \ref{seccontanduccc}) that the
inclusion $A_{cr}^* (X,\mathfrak{U};V) \hookrightarrow A_c^* (X;V)$ induces an 
isomorphism in cohomology provided the space $X$ is contractible. 
For this purpose we consider the abelian groups 
\begin{equation} \label{defrcu}
  A_{cr}^{p,q} ( X,\mathfrak{U} ; V ) := 
\left\{ f: X^{p+1} \times X^{q+1} \rightarrow V 
\mid f_{\mid X^{p+1} \times
      \mathfrak{U}[q]} \; \text{is continuous} \right\}
\, .
\end{equation}
The abelian groups $A_{cr}^{p,q} ( X,\mathfrak{U} ; V )$ 
form a first quadrant double complex whose vertical and horizontal
differentials are given by 
\begin{align*}
 d_{h}^{p,q} : A_{cr}^{p,q} \to A_{cr}^{p+1,q},
& \quad d_{h}^{p,q}(f^{p,q})(\vec{x},\vec{x})
=\sum_{i=0}^{p+1}(-1)^{i}f^{p,q}(x_{0},...,\widehat{x_{i}},...,x_{p+1},\vec{x}')\\
 d_{v}^{p,q} : A_{cr}^{p,q}\to A_{cr}^{p,q+1}, 
&\quad d_{v}^{p,q}(f^{p,q})(\vec{x},\vec{x}')
= (-1)^p
\sum_{i=0}^{q+1}(-1)^{i}f^{p,q}(\vec{x},x_{0}',...,\widehat{x_{i}}',...,x_{q+1}')
\, .
\end{align*}
The double complex $A_{cr}^{*,*} ( X,\mathfrak{U} ; V )$ can be filtrated
column-wise to obtain a spectral sequence $E_{cr,*}^{*,*} (X,\mathfrak{U};V)$ 
(cf. \cite[Theorem 2.15]{Mcl}). 
Since the double complex is a first quadrant double complex, the spectral
sequence $E_{cr,*}^{*,*} (X,\mathfrak{U};V)$ converges to the cohomology of
the total complex of $A_{cr}^{*,*} ( X,\mathfrak{U} ; V )$.

The rows of the double complex $A_{cr}^{*,*} ( X,\mathfrak{U} ; V )$ can be
augmented by the complex $A_{cr}^* (X,\mathfrak{U} ;V)$ for the covering
$\mathfrak{U}$ and the columns can be augmented by the exact complex 
$A_c^* (X;V)$ of continuous cochains:
\begin{equation*}
  \vcenter{
  \xymatrix{
\vdots & \vdots & \vdots & \vdots \\ 
A_{cr}^2 (X, \mathfrak{U};V) \ar[r] \ar[u]_{d_{v}} 
& A_{cr}^{0,2} ( X, \mathfrak{U} ; V) \ar[r]^{d_{h}}\ar[u]_{d_{v}} 
&  A_{cr}^{1,2} ( X, \mathfrak{U}; V) \ar[r]^{d_{h}}\ar[u]_{d_{v}} 
& A_{cr}^{2,2} ( X, \mathfrak{U} ; V) \ar[r]^{d_{h}}\ar[u]_{d_{v}} & \cdots \\
A_{cr}^1 (X, \mathfrak{U};V) \ar[r] \ar[u]_{d_{v}} 
& A_{cr}^{0,1} ( X, \mathfrak{U} ; V) \ar[r]^{d_{h}}\ar[u]_{d_{v}} 
&  A_{cr}^{1,1} ( X, \mathfrak{U} ; V) \ar[r]^{d_{h}}\ar[u]_{d_{v}} 
& A_{cr}^{2,1} ( X, \mathfrak{U} ; V) \ar[r]^{d_{h}}\ar[u]_{d_{v}} & \cdots \\
A_{cr}^0 (X, \mathfrak{U};V) \ar[r] \ar[u]_{d_{v}} 
& A_{cr}^{0,0} ( X, \mathfrak{U} ; V) \ar[r]^{d_{h}}\ar[u]_{d_{v}} 
&  A_{cr}^{1,0} ( X, \mathfrak{U} ; V) \ar[r]^{d_{h}}\ar[u]_{d_{v}} 
&  A_{cr}^{2,0} ( X, \mathfrak{U} ; V) \ar[r]^{d_{h}}\ar[u]_{d_{v}} & \cdots \\
& A_c^0 ( X ; V) \ar[r]^{d_{h}}\ar[u] 
& A_c^1 ( X ; V) \ar[r]^{d_{h}}\ar[u] 
& A_c^2 ( X ; V) \ar[r]^{d_{h}}\ar[u] 
& \cdots
}}
\end{equation*}
We denote the total complex of the double complex 
$A_{cr}^{*,*} ( X, \mathfrak{U} ; V)$ by 
$\tot  A_{cr}^{*,*} ( X, \mathfrak{U} ; V)$. 
The augmentations of the rows and columns of this double complex 
induce morphisms $i^* : A_{cr}^* ( X, \mathfrak{U} ; V) \rightarrow 
\tot A_{cr}^{*,*} ( X, \mathfrak{U} ; V)$ and 
$j^*:  A_c^* ( X ; V) \rightarrow \tot A_{cr}^{*,*} ( X,\mathfrak{U} ; V)$ 
of cochain complexes respectively. 

\begin{lemma} \label{columnsexact}
  The morphism $i^*:  A_{cr}^* ( X,\mathfrak{U} ; V) \rightarrow 
\tot A_{cr}^{*,*} (X,\mathfrak{U} ; V)$ induces an isomorphism in cohomology.
\end{lemma}

\begin{proof}
On each augmented row $A_{cr}^q ( X, \mathfrak{U} ; V) \hookrightarrow 
A_{cr}^{*,q} ( X, \mathfrak{U} ; V)$ 
one can define a contraction $h^{*,q}$ via 
\begin{equation} \label{defrowcontr}
  h^{p,q} : A_{cr}^{p,q} ( X, \mathfrak{U} ; V) \rightarrow 
A_{cr}^{p-1,q} ( X, \mathfrak{U} ; V) , \quad 
h^{p,q} (f) (\vec{x},\vec{x}')= f ( x_0, \ldots, x_{p-1}, x_0 ', \vec{x}' ) \, .
\end{equation}
Therefore the augmented rows are exact and the augmentation $i^*$ induces an 
isomorphism in cohomology.
\end{proof}

\begin{remark}
  Note that for non-trivial $\mathfrak{U}$ this construction does not work for 
the column complexes, because the so constructed cochains would not fulfil the 
continuity condition in Def. \ref{defrcu}.
\end{remark}

For $G$-invariant open coverings $\mathfrak{U}$ of $X$ one can consider the 
sub double complex $A_{cr}^{*,*} ( X,\mathfrak{U} ; V )^G$ of 
$A_{cr}^{*,*} ( X,\mathfrak{U} ; V )$ whose rows are augmented by the cochain 
complex $A_{cr}^* (X,\mathfrak{U} ;V)^G$ for the covering $\mathfrak{U}$ and 
the columns can be augmented by the complex $A_c^* (X;V)^G$ of continuous 
equivariant cochains (,which is not exact in general). 

\begin{lemma} \label{columnsexacteq}
  For $G$-invariant coverings $\mathfrak{U}$ of $X$ the morphism 
$i_{eq}^*:  ={i^*}^G$ induces an isomorphism in cohomology.
\end{lemma}

\begin{proof}
  The contraction $h_{*,q}$ of the augmented rows 
$A_{cr}^q ( X, \mathfrak{U} ; V) \hookrightarrow 
\tot A_{cr}^{*,q} ( X, \mathfrak{U} ; V)$ defined in Eq. \ref{defrowcontr} is
$G$-equivariant and thus restricts to a row contraction of the augmented
sub-row $A_{cr}^q ( X, \mathfrak{U} ; V)^G \hookrightarrow 
\tot A_{cr}^{*,q} ( X, \mathfrak{U} ; V)^G$.
\end{proof}

So the morphism 
$H (i_{eq}) : H_{cr,eq} (X,\mathfrak{U};V) \rightarrow 
H ( \tot A_{cr}^{*,*} ( X, \mathfrak{U} ; V)^G )$ is invertible. 
For the composition 
$H (i_{eq})^{-1} H(j_{eq}):H_{c,eq}(X;V)\rightarrow H_{cr,eq}(X,\mathfrak{U};V)$ 
we observe:

\begin{proposition} \label{contiscohtocr}
The image $j^n (f)$ of a continuous equivariant $n$-cocycle $f$ on $X$ in 
$\tot A_{cr}^{*,*} (X,\mathfrak{U};,V)^G$ is cohomologous to the image 
$i_{eq}^n  (f)$ of the equivariant $n$-cocycle 
$f\in A_{cr}^n (X,\mathfrak{U};V)^G$ in $\tot A_{cr}^{*,*} (X,\mathfrak{U};V)^G$.
\end{proposition}

\begin{proof}
The proof is a variation of the proof of \cite[Proposition 14.3.8]{F10}:
 Let $f:X^{n+1}\to V$ be a continuous equivariant $n$-cocycle on 
 $X$ and (for $p+q=n-1$) define equivariant cochains 
$\psi^{p,q} : X^{p+1}\times X^{q+1} \cong X^{n+1} \to V$ in 
$A_{cr}^{p,q} (X,\mathfrak{U};V)^G$ via 
$\psi^{p,q} ( \vec{x},\vec{x}')= (-1)^p f (\vec{x},\vec{x}')$. The vertical 
coboundary of the cochain $\psi^{p,q}$ is given by
 \begin{eqnarray*}
  [d_v \psi^{p,q}] (\vec{x},x_0',\ldots,x_{q+1}') & = & (-1 )^p
  \sum (-1)^i f (\vec{x},x_0',\ldots,\hat{x}_i',\ldots,x_q') \\
  & = & - \sum (-1)^{p+1+i} f ( x_0,\ldots,\hat{x}_i, \ldots,
  x_p,\vec{x}') \\
  & = & [d_h \psi^{p-1,q+1}](x_0,...,x_p,\vec{x}').
 \end{eqnarray*}
 The anti-commutativity of the horizontal and the vertical differential 
 ensures that the coboundary of the cochain
 $\sum_{p+q=n-1} (-1)^p \psi^{p,q}$ in the total complex 
 is the cochain $j^n (f) - i^n (f)$. Thus the cocycles $j^n (f)$ 
 and $i_{eq}^n (f)$ are cohomologous in $\tot A_{cr}^{*,*} (X,\mathfrak{U};V)^G$.
\end{proof}

\begin{corollary}
 The composition 
$H (i_{eq})^{-1} H(j_{eq}):H_{c,eq}(X;V)\rightarrow H_{cr,eq}(X,\mathfrak{U};V)$ 
is induced by the inclusion 
$A_c^* (X,\mathfrak{U};V)^G \hookrightarrow A_{cr}^* (X,\mathfrak{U};V)^G$. 
\end{corollary}

\begin{corollary}
If the morphism 
$j^*_{eq}:={j^*}^G : A_c^* (X;V)^G \rightarrow 
\tot A_{cr}^{*,*}(X,\mathfrak{U}A)^G$ induces a monomorphism, epimorphism or
isomorphism in cohomology, then the inclusion 
$A_c^* (X;V)^G \hookrightarrow A_{cr}^* (X,\mathfrak{U};V)^G$
induces a monomorphism, epimorphism or isomorphism in cohomology respectively.
\end{corollary}
 
For any directed system $\{ \mathfrak{U}_i \mid i \in I \}$ of open coverings 
of $X$ one can also consider the corresponding augmented colimit double
complexes. In particular for the directed system of all open coverings of $X$ 
one obtains the double complex complex 
\begin{equation*}
  A_{cg}^{*,*} (X;V):= \colim_{\mathfrak{U} \text{ is open cover of $X$}} 
A_{cr}^{*,*} (X;\mathfrak{U};V)
\end{equation*}
whose rows and columns are augmented by the colimit complex 
$A_{cg}^* (X;V)$ and by the complex $A_c^* (X;V)$ respectively. 

\begin{lemma}
  For any directed system $\{ \mathfrak{U}_i \mid i \in I \}$ of open 
coverings of $X$ the morphism 
$\colim_i i^*:  \colim_i A_{cr}^* ( X,\mathfrak{U}_i ; V) \rightarrow 
\tot \colim_i A_{cr}^{*,*} (X,\mathfrak{U}_i ; V)$ induces an isomorphism in 
cohomology.
\end{lemma}

\begin{proof}
The passage to the colimit preserves the exactness of the augmented row
complexes (Lemma \ref{columnsexact}). 
\end{proof}

As a consequence the colimit morphism 
$i_{cg}^* : A_{cg}^* ( X; V) \rightarrow \tot A_{cg}^{*,*} (X; V)$ induces an 
isomorphism in cohomology.
The colimit double complex $A_{cg}^{*,*} (X;V)$ is a double complex of 
$G$-modules and the $G$-equivariant cochains in form a sub double complex 
$A_{cg}^{*,*} (X;V)^G$, whose rows and columns are augmented by the colimit 
complex $A_{cg,eq}^* (X;V)$ and by the complex $A_c^* (X;V)^G$ respectively.

\begin{lemma}
  For any directed system $\{ \mathfrak{U}_i \mid i \in I \}$ of $G$-invariant
  open coverings of $X$ the morphism 
$\colim_i i_{eq}^*:  \colim_i A_{cr}^* ( X,\mathfrak{U}_i ; V)^G \rightarrow 
\tot \colim_i A_{cr}^{*,*} (X,\mathfrak{U}_i ; V)^G$ induces an isomorphism in 
cohomology.
\end{lemma}

\begin{proof}
The passage to the colimit preserves the exactness of the augmented row
complexes (Lemma \ref{columnsexacteq}). 
\end{proof}

Moreover, since the open diagonal neighbourhoods $\mathfrak{U}[n]$ in 
$X^{n+1}$ for open coverings $\mathfrak{U}$ of $X$ are cofinal in the 
directed set of all open diagonal neighbourhoods, we observe:

\begin{lemma} \label{natinclofeqdcisiso}
  The natural morphism of double complexes 
  \begin{equation*}
A_{cg,eq}^{*,*} (X;V):= \colim_{\mathfrak{U} \text{is $G$-invariant open cover
      of $X$}} A_{cr}^{*,*} (X;\mathfrak{U};V)^G \rightarrow A_{cg}^* (X;V)^G    
  \end{equation*}
is a natural isomorphism.
\end{lemma}

\begin{proof}
  The proof is analogous to that of Proposition \ref{natinclofeqccisiso}.
\end{proof}

As a consequence the colimit morphism 
$i_{cg,eq}^* : A_{cg,eq}^* ( X; V) \rightarrow \tot A_{cg}^{*,*} (X; V)^G$ 
induces an isomorphism in cohomology, and the morphism 
$H (i_{cg,eq})$ is invertible. For the composition 
$H(i_{cg,eq})^{-1} H(j_{eq}):H_{c,eq}(X;V)\rightarrow H_{cg,eq}(X,\mathfrak{U};V)$ 
we observe:

\begin{proposition} \label{contiscohtocreq}
The image $j^n (f)$ of a continuous equivariant $n$-cocycle $f$ on $X$ in 
$\tot A_{cg}^{*,*} (X;,V)^G$ is cohomologous to the image 
$i_{cg,eq}^n  (f)$ of the equivariant $n$-cocycle $f\in A_{cg,eq}^n (X;V)$ 
in $\tot A_{cg}^{*,*} (X;V)^G$.
\end{proposition}

\begin{proof}
  The proof is analogous to that of Proposition \ref{contiscohtocr}.
\end{proof}

\begin{corollary}
 The composition 
$H (i_{cg,eq})^{-1} H(j_{eq}):H_{c,eq}(X;V)\rightarrow H_{cg,eq}(X;V)$ 
is induced by the inclusion 
$A_c^* (X;V)^G \hookrightarrow A_{cg}^* (X;V)^G$. 
\end{corollary}

\begin{corollary}
If the morphism 
$j^*_{eq}:={j^*}^G : A_c^* (X;V)^G \rightarrow \tot A_{cg}^{*,*}(X;V)^G$
induces a monomorphism, epimorphism or isomorphism in cohomology, then the 
inclusion $A_c^* (X;V)^G \hookrightarrow A_{cg,eq}^* (X;V)$ induces a 
monomorphism, epimorphism or isomorphism in cohomology respectively.
\end{corollary}

\section{Continuous and $\mathfrak{U}$ -Continuous Cochains}
\label{seccontanduccc}

In this section we consider transformation groups $(G,X)$ and $G$-modules $V$
for which we show that the inclusion 
$A_c^* (X,\mathfrak{U};V)^G \hookrightarrow A_{cr}^* (X,\mathfrak{U};V)^G$
of the complex of continuous equivariant cochains into the complex 
of equivariant $\mathfrak{U}$-continuous cochains induces an isomorphism 
$H_c^* (X,\mathfrak{U};V) \cong H_{cr}^* (X,\mathfrak{U};V)$ provided the
topological space $X$ is contractible.
The proof relies on the row exactness of the double complexes 
$A_c^{*,*} (X,\mathfrak{U};V )^G$ and $A_{cr}^{*,*} ( X,\mathfrak{U};V)^G$. 
At first we reduce the problem to the non-equivariant case:

\begin{proposition} \label{noneqextheneqex}
 If the augmented column complexes 
$A_c^p (X;V) \hookrightarrow A_{cr}^{p,*}(X,\mathfrak{U};V)$ 
are exact, then the augmented sub column complexes
$A_c^p (X;V)^G \hookrightarrow A_{cr}^{p,*}(X,\mathfrak{U};V)^G$  
of equivariant cochains are exact as well.
\end{proposition}

\begin{proof}
 Assume that the augmented column complexes 
$A_c^p (X,\mathfrak{U};V) \hookrightarrow A_{cr}^{p,*}(X,\mathfrak{U};V)$ 
are exact. Then each equivariant vertical cocycle $f_{eq}^{p,q} \in A_{rc}^{p,q}
(X,\mathfrak{U};V)^G$ is the vertical coboundary $d_v f^{p,q-1}$ of a 
cocycle $f^{p,q-1} \in A_{cr}^{p,q-1} (X,\mathfrak{U};V)$ 
(which is not necessary equivariant).
Define an equivariant cochain $f_{eq}^{p,q-1}$ of bidegree \mbox{$(p,q-1)$ via}
 \begin{equation*}
  f_{eq}^{p,q-1} (\vec{x},\vec{x}'):=
  x_0 . f^{p,q-1} (x_0^{-1} . \vec{x}, x_0^{-1} . \vec{x}') \, .
 \end{equation*}
This equivariant cochain is continuous on $X^{p+1} \times \mathfrak{U}[q-1]$ 
because $f^{p,q-1}$ is continuous on $X^{p+1} \times \mathfrak{U}[q-1]$. 
We assert that the vertical coboundary $d_v  f_{eq}^{p,q-1}$ of $f_{eq}$ is
the equivariant vertical cocycle $f_{eq}^{p,q}$. 
Indeed, since the differential $d_v$ is equivariant, the vertical coboundary 
of $f_{eq}^{p,q-1}$ computes to
\begin{equation*}
  d_v f_{eq}^{p,q-1} (\vec{x},\vec{x}') = 
  x_0 . \left[ d_v f^{p,q-1} (x_0^{-1} . \vec{x} , x_0^{-1} . \vec{x}')
  \right]=
  f_{eq}^{p,q} (\vec{x} , \vec{x}') \, .
 \end{equation*}
 Thus every equivariant vertical cocycle $f_{eq}^{p,q}$ in 
$A_{cr}^{*,*} (X,\mathfrak{U};V)^G$ is the vertical coboundary of an 
equivariant cochain $f_{eq}^{p,q-1}$ of bidegree $(p,q-1)$.
\end{proof}

\begin{corollary} \label{augexthenjeqindiso}
If the augmented column complexes 
$A_c^p (X;V) \hookrightarrow A_{cr}^{p,*} (X,\mathfrak{U};V)$ 
are exact, then the inclusion 
$j_{eq}^* : A_c^* (X;V)^G \hookrightarrow 
\tot A_{cr}^{*,*}(X,\mathfrak{U},V)^G$ induces an isomorphism in cohomology.
\end{corollary}

\begin{corollary} \label{augextheninclindiso}
If the augmented column complexes 
 $A_c^p (X;V) \hookrightarrow A_{cr}^{p,*} (X,\mathfrak{U};V)$ 
are exact, then the inclusion 
$A_c^* (X;V)^G \hookrightarrow A_{cr}^* (X,\mathfrak{U};V)^G$
induces an isomorphism in cohomology.
\end{corollary}

To achieve the announced result it remains to show that for 
contractible spaces $X$ the colimit augmented columns 
$A_c^p (X;V) \hookrightarrow A_{cg}^{p,*} (X;V)$ are exact. For this purpose
we first consider the cochain complex associated to the cosimplicial abelian 
group 
$A^{p,*} (X;V) := \left\{ f : X^{p+1} \times X^{*+1} \rightarrow V 
\mid \, \forall \vec{x}' \in X^{*+1} : f (-,\vec{x}') \in C (X^{p+1},V) 
\right\}$ 
of global cochains, its subcomplex $A_{cr}^{p,*} (X,\mathfrak{U};V)$ and 
the cochain complexes associated to the cosimplicial abelian groups 
\begin{eqnarray*}
 A^{p,*} (\mathfrak{U};V) & := & 
\{ f : X^{p+1} \times \mathfrak{U}[*] \rightarrow 
\mid \, \forall \vec{x}' \in \mathfrak{U}[*] : f (-,\vec{x}') \in C (X^{p+1},V)
\} \quad \text{and} \\ 
A_c^{p,*} (X,\mathfrak{U};V) & := & C ( X^{p+1} \times \mathfrak{U}[*] , V) \, .
\end{eqnarray*}

Restriction of global to local cochains induces morphisms of cochain complexes 
$\res^{p,*} : A^{p,*} (X;V) \twoheadrightarrow  A^{p,*} (X,\mathfrak{U};V)$ and 
$\res_{cr}^{p,*} : A_{cr}^* (X,\mathfrak{U};V) \twoheadrightarrow  
A_c^{p,*} (X,\mathfrak{U};V)$ 
intertwining the inclusions of the subcomplexes 
$A_{cr}^{p,*} (X,\mathfrak{U};V) \hookrightarrow A^{p,*} (X;V)$ and 
$A_c^{p,*} (X,\mathfrak{U};V) \hookrightarrow A^{p,*} (X, \mathfrak{U};V)$, so
one obtains the following commutative diagram
\begin{equation} \label{morphexseq}
\begin{array}{cccccccc}
 0 \longrightarrow & \ker (\res_{cr}^{p,*} ) & \longrightarrow & 
A_{cr}^{p,*} (X,\mathfrak{U};V) 
& \longrightarrow & A_c^{p,*} (X,\mathfrak{U};V) & \longrightarrow 0 \\ 
& \downarrow & & \downarrow & & \downarrow & \\
0 \longrightarrow & \ker (\res^{p,*} ) & \longrightarrow & A^{p,*} (X;V) 
& \longrightarrow & A^{p,*} (X , \mathfrak{U};V) & \longrightarrow 0
\end{array}
\end{equation}
of cochain complexes whose rows are exact. The kernel $\ker (\res^{p,q} )$ is
the subspace of those $(p,q)$-cochains which are trivial on 
$X^{p+1} \times \mathfrak{U} [q]$. Since these $(p,q)$-cochains are continuous 
on $X^{p+1} \times \mathfrak{U}[q]$ we find that both kernels coincide. 
We abbreviate the complex 
$\ker (\res^{p,*} ) = \ker (\res_{rc}^{p,*} )$ by $K^{p,*}$ and denote the 
cohomology groups of the complex $A_{cr}^{p,*} (X,\mathfrak{U};V)$ by 
$H_{cr}^{p,*} (X,\mathfrak{U};V)$, the cohomology groups of the complex 
$A_c^{p,*} (X,\mathfrak{U};V)$ of continuous cochains by 
$H_c^{p,*} (X,\mathfrak{U};V)$ and the cohomology groups of the complex 
$A^{p,*} (X,\mathfrak{U};V)$ by $H^{p,*} (X,\mathfrak{U};V)$. 

\begin{lemma}
  The cochain complexes $A^{p,*} (X;V)$ are exact.
\end{lemma}

\begin{proof}
For any point $* \in X$ the homomorphisms 
$h^{p,q} : A^{p,q} (X;V) \rightarrow A^{p,q-1} (X;V)$ given by 
$h^{p,q} (f) (\vec{x},\vec{x}'):=f (\vec{x},*,\vec{x}')$ form a contraction of
the complex $A^{p,*} (X;V)$. 
\end{proof}

The morphism of short exact sequences of cochain complexes in Diagram 
\ref{morphexseq}
gives rise to a morphism of long exact cohomology sequences, in which the 
cohomology of the complex $A^{p,*} (X;V)$ is trivial:
\begin{equation} \label{diaglecs}
\xymatrix@R-10pt@C-4pt{ 
\ar[r] & H^q (K^{p,*} )) \ar[r] \ar@{=}[d] 
& H_{cr}^{p,q} (X,\mathfrak{U};V) \ar[r] \ar[d]  
& H_c^{p,q} (X,\mathfrak{U};V) \ar[r] \ar[d]  
& H^{q+1}  (K^{p,*} ) \ar[r] \ar@{=}[d] & {} \\ 
\ar[r]^\cong & H^q (K^{p,*} ) \ar[r]&  0  \ar[r] 
& H^{p,q} (X,\mathfrak{U};V) \ar[r]^\cong &  H^{q+1}  (K^{p,*} )) \ar[r] & {}
}
\end{equation}

\begin{lemma}
The augmented complex 
$A_c^p (X;V) \hookrightarrow A_{cr}^{p,*} (X,\mathfrak{U};V)$ is exact if and
only if the inclusion 
$A_c^{p,*} (X,\mathfrak{U};V) \hookrightarrow A^{p,*} (X,\mathfrak{U};V)$ 
induces an isomorphism in cohomology.
\end{lemma}

\begin{proof}
 This is an immediate consequence of Diagram \ref{diaglecs}
\end{proof}

\begin{proposition}
  If the inclusion 
$A_c^{p,*} (X,\mathfrak{U};V) \hookrightarrow A^{p,*} (X,\mathfrak{U};V)$ 
induces an isomorphism in cohomology, then the inclusions  
$j_{eq}^* : A_c^* (X;V)^G \hookrightarrow 
\tot A_{cr}^{*,*}(X,\mathfrak{U},V)^G$ and 
$A_c^* (X,\mathfrak{U};V)^G \hookrightarrow A_{cr}^* (X,\mathfrak{U};V)^G$
also induces an isomorphism in cohomology.
\end{proposition}

\begin{proof}
  This follows from the preceding Lemma and Corollaries 
\ref{augexthenjeqindiso} and \ref{augextheninclindiso}.
\end{proof}
The passage to the colimit over all open coverings of $X$ yields the
corresponding results for the complexes of cochains with continuous germs: 

\begin{proposition} \label{noneqextheneqexcg}
 If the augmented column complexes 
$A_c^p (X;V) \hookrightarrow A_{cg}^{p,*}(X;V)$ 
are exact, then the augmented sub column complexes
$A_c^p (X;V)^G \hookrightarrow A_{cg}^{p,*}(X;V)^G$  
of equivariant cochains are exact as well.
\end{proposition}

\begin{proof}
  The proof is similar to that of Proposition \ref{noneqextheneqex}.
\end{proof}

\begin{corollary} \label{augexthenjeqindisocg}
If the augmented column complexes 
$A_c^p (X;V) \hookrightarrow A_{cg}^{p,*} (X;V)$ 
are exact, then the inclusion 
$j_{eq}^* : A_c^* (X;V)^G \hookrightarrow 
\tot A_{cg}^{*,*}(X;V)^G$ induces an isomorphism in cohomology.
\end{corollary}

\begin{corollary} \label{augextheninclindisocg}
If the augmented column complexes 
 $A_c^p (X;V) \hookrightarrow A_{cg}^{p,*} (X;V)$ 
are exact, then the inclusion 
$A_c^* (X;V)^G \hookrightarrow A_{cg}^* (X;V)^G$
induces an isomorphism in cohomology.
\end{corollary}

\begin{remark} \label{remonlyginvcov}
Alternatively to taking the colimit over all open coverings $\mathfrak{U}$ of 
$X$ one may consider $G$-invariant open coverings only to obtains the same
results. 
(This was shown in Proposition \ref{natinclofeqccisiso} and Lemmata 
\ref{natinclofeqdcisiso}.)
\end{remark}

\begin{example}
  If $G=X$ is a topological group which acts on itself by left translation and 
the augmented columns  
$A_c^p (X;V) \hookrightarrow A_{cg}^{p,*} (X;V) := 
\colim A^{p,*} (X,\mathfrak{U}_U;V)$ (where $U$ runs over all open identity
neighbourhoods in $G$)  
are exact, then the inclusion 
$A_c^* (X;V)^G \hookrightarrow A_{cg}^* (X;V)^G$
induces an isomorphism in cohomology.
\end{example}

The complex $A^{p,*} (X,\mathfrak{U};V)$ is isomorphic to the complex 
$A^* (\mathfrak{U}; C (X^{p+1},V))$. The colimit 
$A_{AS}^* (X; C (X^{p+1} ,V)):=\colim A^* (\mathfrak{U}; C (X^{p+1} ,V))$, 
where $\mathfrak{U}$ runs over all open coverings of
$X$ is the complex of Alexander-Spanier cochains on $X$. Therefore the colimit 
complex $\colim A^p (X; A^* (\mathfrak{U};V))$ is isomorphic to the cochain 
complex $A_{AS}^* (X;C (X^{p+1} ,V))$. 
A similar observation can be made for the cochain complex 
$A_c^{p,*} (X,\mathfrak{U};V)$ if the exponential law 
$C (X^{p+1} \times \mathfrak{U}[q],V) \cong C (X,C (\mathfrak{U}[q],V))$ 
holds for a cofinal set of open coverings $\mathfrak{U}$ of $X$. 
Passing to the colimit in Diagram \ref{morphexseq} yields the morphism

\begin{equation} \label{morphexseqcg}
\begin{array}{cccccccc}
 0 \longrightarrow & \ker (\res_{cg}^{p,*} ) & \longrightarrow & 
A_{cg}^{p,*} (X;V) & \longrightarrow & \colim A_c^{p,*} (X,\mathfrak{U};V) & 
\longrightarrow 0 \\ 
& \downarrow & & \downarrow & & \downarrow & \\
0 \longrightarrow & \ker (\res^{p,*} ) & \longrightarrow & A^{p,*} (X;V) 
& \longrightarrow & A_{AS}^* (X; C^{p+1} (X,V))
& \longrightarrow 0
\end{array}
\end{equation}
of short exact sequences of cochain complexes. The kernel $\ker (\res^{p,q})$
is the subspace of those $(p,q)$-cochains which are trivial on 
$X^{p+1} \times \mathfrak{U} [q]$ for some open covering $\mathfrak{U}$ of 
$X$. Since these $(p,q)$-cochains are continuous on 
$X^{p+1} \times \mathfrak{U}[q]$ we find that both kernels coincide. We 
abbreviate the complex $\ker (\res^{p,*} ) = \ker (\res_{cg}^{p,*} )$ by 
$K_{cg}^{p,*}$ and denote the cohomology groups of the complex 
$A_{cg}^{p,*} (X;V)$ by $H_{cg}^{p,*} (X;V)$. 
The morphism of short exact sequences of cochain complexes in Diagram 
\ref{morphexseqcg} gives rise to a morphism of long exact cohomology 
sequences:
\begin{equation} \label{diaglecscg}
\xymatrix@C-11pt@R-10pt{ 
\ar[r] & H^q (K_{cg}^{p,*} )) \ar[r] \ar@{=}[d] 
& H_{cg}^{p,q} (X,\mathfrak{U};V) \ar[r] \ar[d]  
& H^q (\colim A_c^{p,*} (X,\mathfrak{U};V) \ar[r] \ar[d]  
& H^{q+1}  (K_{cg}^{p,*} ) \ar[r] \ar@{=}[d] & {} \\ 
\ar[r]^\cong & H^q (K^{p,*} ) \ar[r]&  0  \ar[r] 
& H_{AS}^q (X; C^{p+1} (X,V)) \ar[r]^\cong &  H^{q+1}  (K^{p,*} )) \ar[r] & {}
}
\end{equation}

\begin{lemma}
The augmented complex 
$A_c^p (X;V) \hookrightarrow A_{cg}^{p,*} (X;V)$ is exact if and only if 
the inclusion 
$\colim A_c^{p,*} (X,\mathfrak{U};V)\hookrightarrow A_{AS}^* (X;C^{p+1}(X,V))$
of cochain complexes induces an isomorphism in cohomology.
\end{lemma}

\begin{proof}
 This is an immediate consequence of Diagram \ref{diaglecscg}
\end{proof}

\begin{proposition} \label{inclcolimacpsindisothenjeqiso}
  If the inclusion 
$\colim A_c^{p,*} (X,\mathfrak{U};V)\hookrightarrow A_{AS}^* (X;C(X^{p+1},V))$ 
induces an isomorphism in cohomology, then  
$j_{eq}^* : A_c^* (X;V)^G \hookrightarrow \tot A_{cg}^{*,*}(X;V)^G$ and 
$A_c^* (X;V)^G \hookrightarrow A_{cg}^* (X;V)^G$
also induce an isomorphism in cohomology.
\end{proposition}

\begin{proof}
  This follows from the preceding Lemma and Corollaries 
\ref{augexthenjeqindisocg} and \ref{augextheninclindisocg}.
\end{proof}

As observed before (cf. Remark \ref{remonlyginvcov}) one may restrict oneself
to the directed system of $G$-invariant open coverings only to achieve the
same result. Thus we observe:

\begin{corollary}
  If $G=X$ is a locally contractible topological group which acts on itself 
by left translation and the inclusion 
$\colim A_c^{p,*} (X,\mathfrak{U};V)\hookrightarrow A_{AS}^* (X;C(X^{p+1},V))$ 
(where $U$ runs over all open identity
neighbourhoods in $G$)  induces an isomorphism in cohomology, then the 
inclusion $A_c^* (X;V)^G \hookrightarrow A_{cg}^* (X;V)^G$
induces an isomorphism in cohomology as well.
\end{corollary}

\begin{proof}
 It has been shown in \cite{vE62b} that the cohomology of the colimit cochain 
complex 
$\colim A^* ( \mathfrak{U} ;C(X^{p+1},V))$ is the Alexander-Spanier cohomology
of $X$.  
\end{proof}

\begin{lemma} \label{xcontrthenacpsistriv}
  If the topological space $X$ is contractible, then the cohomology 
of the complex $\colim A_c^{p,*} (X,\mathfrak{U};V)$ is trivial.
\end{lemma}

\begin{proof}
The reasoning is analogous to that for the Alexander-Spanier presheaf. 
The proof \cite[Theorem 2.5.2]{F10} carries over almost in verbatim. 
\end{proof}

\begin{theorem}
For contractible $X$ the inclusion 
$A_c^* (X;V)^G \hookrightarrow A_{cg}^* (X;V)^G$ 
induces an isomorphism in cohomology.
\end{theorem}

\begin{proof}
If the topological space $X$ is contractible, then the Alexander-Spanier
cohomology $H_{AS} (X;C^{p+1}(X,V))$ is trivial and the cohomology of the
cochain complex $\colim A_c^{p,*} (X,\mathfrak{U};V)$ is trivial by Lemma 
\ref{xcontrthenacpsistriv}. By Proposition \ref{inclcolimacpsindisothenjeqiso}
the inclusion $A_c^* (X;V)^G \hookrightarrow A_{cg}^* (X;V)^G$ then induces an
isomorphism in cohomology.
\end{proof}

\begin{corollary} \label{gcontriso}
  For contractible topological groups $G$ the 
continuous group cohomology is isomorphic to the
cohomology of homogeneous group cochains with continuous germ at the diagonal.
\end{corollary}

\section{working in the category of  $k$-spaces}
\label{secwinktop}

In this section we consider transformation groups $(G,X)$ in the category
$\ktop$ of $k$-spaces and $G$-modules $V$ in $\ktop$. Working only in the
category $\ktop$ we construct a spectral sequence analogously to that in
Section \ref{sectss} and derive results analogous to those obtained there.

\begin{definition}
For every $k$-space $X$ and abelian $k$-group $V$ the subcomplex 
$A_{kc}^* (X;V):= C ( \fktop X^{*+1};V)$ of the standard complex is called the 
\emph{continuous standard complex in $\ktop$}. 
\end{definition}

For open coverings $\mathfrak{U}$ of a $k$-space $X$ we also consider the 
subcomplex of $A^* (X;V)$ formed by the groups  
\begin{equation*}
A_{kcr}^n (X,\mathfrak{U};V) := 
\left\{ f \in  A^n (X;V) \mid \, f_{\mid \fktop \mathfrak{U} [n]}
  \in C ( \fktop \mathfrak{U}[n];V) \right\}  
\end{equation*}
of cochains whose restriction to the open subspaces 
$\fktop \mathfrak{U}[n]$ of $\fktop X^{n+1}$ are continuous. 
The cohomology of the cochain complex 
$A_{kcr}^* (X,\mathfrak{U};V)$ is denoted by $H_{kcr} (X,\mathfrak{U};V)$. 
If the covering $\mathfrak{U}$ of $X$ is $G$-invariant, then the subspaces 
$\fktop \mathfrak{U}[*]$ is a simplicial $G$-subspace of the simplicial 
$G$-space $\fktop X^{*+1}$. 

\begin{example}
  If $G=X$ is a $k$-group which acts on itself by left translation and
  $U$ an open identity neighbourhood, then 
$\mathfrak{U}_U :=\{ g.U \mid g \in G \}$ is a $G$-invariant open covering of 
$G$ and $\fktop \mathfrak{U}[*]$ is a simplicial $G$-subspace of 
$\fktop G^{*+1}$.
\end{example}

For $G$-invariant coverings $\mathfrak{U}$ of $X$ the cohomology of the 
subcomplex $A_{kcr}^* (X,\mathfrak{U};V)^G$ of $G$-equivariant cochains is 
denoted by $H_{kcr,eq} (X,\mathfrak{U};V)$. 

\begin{example}
  If $G=X$ is a $k$-group which acts on itself by left translation and
  $U$ an open identity neighbourhood, then the complex 
$A_{kcr}^* (X,\mathfrak{U}_U;V)^G$ is the complex of homogeneous group cochains 
whose restrictions to the subspaces $\fktop \mathfrak{U}_U [*]$ are continuous. 
(These are sometimes called $\mathfrak{U}$-continuous cochains.) 
\end{example}

For directed systems $\{ \mathfrak{U}_i \mid i \in I \}$ of open coverings of
$X$ one can also consider the colimit complex 
$\colim_i A_{kcr}^* (X,\mathfrak{U}_i ;V)$. In particular, if the open diagonal
neighbourhoods  $\fktop \mathfrak{U}[n]$ in $\fktop X^{n+1}$ for open coverings 
$\mathfrak{U}$ of $X$ are cofinal in the directed set of all open diagonal 
neighbourhoods, one obtains the complex
\begin{equation*}
  A_{kcg}^* (X;V):= \colim_{\mathfrak{U} \text{is open cover of $X$}} A_{kcr}^*
  (X;\mathfrak{U};V)
\end{equation*}
of global cochains whose germs at the diagonal are continuous. This happens
for all $k$-spaces $X$ for which the finite products $X^{n+1}$ in $\tops$ are
already compactly Hausdorff generated, e.g. metrisable spaces, locally compact
spaces or Hausdorff $k_\omega$-spaces. 
The complex $A_{kcg}^* (X;V)$ is then a subcomplex of the standard complex 
$A^* (X;V)$ which is invariant under the $G$-action (Eq. \ref{defgact}) and 
thus a sub complex of $G$-modules. 
The $G$-equivariant cochains with continuous germ form a subcomplex 
$A_{kcg}^* (X;V)^G$ thereof, whose cohomology is denoted by $H_{kcg,eq} (X;V)$. 
The latter subcomplex can also be obtained by taking the colimit over all 
$G$-invariant open coverings of $X$ only: 

\begin{proposition} \label{natinclofeqccisisok}
If the open diagonal neighbourhoods $\fktop \mathfrak{U}[n]$ in 
$\fktop X^{n+1}$ for open coverings $\mathfrak{U}$ of $X$ are cofinal in the 
directed set of all open diagonal neighbourhoods then the natural morphism of 
cochain complexes 
  \begin{equation*}
A_{kcg,eq}^* (X;V):= \colim_{\mathfrak{U} \text{is $G$-invariant open cover
      of $X$}} A_{kcr}^* (X;\mathfrak{U};V)^G \rightarrow A_{kcg}^* (X;V)^G    
  \end{equation*}
is a natural isomorphism.
\end{proposition}

\begin{proof}
  The proof is analogous top that of Proposition \ref{natinclofeqccisiso}.
\end{proof}

\begin{corollary}
If the open diagonal neighbourhoods  $\fktop \mathfrak{U}[n]$ in 
$\fktop X^{n+1}$ for open coverings $\mathfrak{U}$ of $X$ are cofinal in the 
directed set of all open diagonal neighbourhoods then the cohomology 
$H_{kcg,eq} (X;V)$ is the 
cohomology of the complex of equivariant cochains which are continuous on 
some $G$-invariant neighbourhood of the diagonal.
\end{corollary}

\begin{example}
If $G=X$ is a metrisable or locally compact topological group or a real or
complex Kac-Moody group which acts on itself by left 
translation, then the complex $A_{kcg}^* (G;V)^G$ is the complex of 
homogeneous group cochains whose germs at the diagonal are continuous. 
(By abuse of language these are sometimes called 'locally continuous' 
group cochains.) 
\end{example}

Analogously to the procedure in Section \ref{sectss} we can construct a
spectral sequence relating $A_{kcr}^* (X,\mathfrak{U};V)$ and 
$A_{kc}^* (X;V)$. For this purpose we consider the abelian groups 
\begin{equation} \label{defrcuk}
  A_{kcr}^{p,q} ( X,\mathfrak{U} ; V ) := 
\left\{ f: X^{p+1} \times X^{q+1} \rightarrow V 
\mid f_{\mid \fktop X^{p+1} \times_k
      \fktop \mathfrak{U}[q]} \; \text{is continuous} \right\}
\, .
\end{equation}
The abelian groups $A_{kcr}^{p,q} ( X,\mathfrak{U} ; V )$ 
form a first quadrant double complex whose vertical and horizontal
differentials are given by the same formulas as for the double complex 
$A_{cr}^{p,q} ( X,\mathfrak{U} ; V )$ introduced in Section \ref{sectss}. 
Analogously to the latter double complex the rows of the double complex 
$A_{kcr}^{*,*} ( X,\mathfrak{U} ; V )$ can be
augmented by the complex $A_{kcr}^* (X,\mathfrak{U} ;V)$ for the covering
$\mathfrak{U}$ and the columns can be augmented by the exact complex 
$A_{kc}^* (X;V)$ of continuous cochains. We denote the total complex of the 
double complex $A_{kcr}^{*,*} ( X, \mathfrak{U} ; V)$ by 
$\tot  A_{kcr}^{*,*} ( X, \mathfrak{U} ; V)$. 
The augmentations of the rows and columns induce morphisms 
$i_k^* : A_{kcr}^* ( X, \mathfrak{U} ; V) \rightarrow 
\tot A_{kcr}^{*,*} ( X, \mathfrak{U} ; V)$ and 
$j_k^*:  A_{kc}^* ( X ; V) \rightarrow \tot A_{kcr}^{*,*} ( X,\mathfrak{U} ; V)$ 
of cochain complexes respectively. 

\begin{lemma} \label{columnsexactk}
  The morphism $i_k^*:  A_{kcr}^* ( X,\mathfrak{U} ; V) \rightarrow 
\tot A_{kcr}^{*,*} (X,\mathfrak{U} ; V)$ induces an isomorphism in cohomology.
\end{lemma}

\begin{proof}
  The proof of Lemma \ref{columnsexact} also works in the category $\ktop$ of
  $k$-spaces.
\end{proof}

For $G$-invariant open coverings $\mathfrak{U}$ of $X$ one can consider the 
sub double complex $A_{kcr}^{*,*} ( X,\mathfrak{U} ; V )^G$ of 
$A_{kcr}^{*,*} ( X,\mathfrak{U} ; V )$ whose rows are augmented by the cochain 
complex $A_{kcr}^* (X,\mathfrak{U} ;V)^G$ for the covering $\mathfrak{U}$ and 
the columns can be augmented by the complex $A_{kc}^* (X;V)^G$ of continuous 
equivariant cochains (,which is not exact in general). 

\begin{lemma} \label{columnsexacteqk}
  For $G$-invariant coverings $\mathfrak{U}$ of $X$ the morphism 
$i_{k,eq}^*:  ={i_k^*}^G$ induces an isomorphism in cohomology.
\end{lemma}

\begin{proof}
 The proof is analogous to that of Lemma \ref{columnsexacteq}.
\end{proof}

So the morphism 
$H (i_{k,eq}) : H_{kcr,eq} (X,\mathfrak{U};V) \rightarrow 
H ( \tot A_{kcr}^{*,*} ( X, \mathfrak{U} ; V)^G )$ is invertible. 
For the composition 
$H (i_{k, eq})^{-1} H(j_{k,eq}):H_{kc,eq}(X;V)\rightarrow 
H_{kcr,eq}(X,\mathfrak{U};V)$ 
we observe:

\begin{proposition} \label{contiscohtocrk}
The image $j_k^n (f)$ of a continuous equivariant $n$-cocycle $f$ on $X$ in 
$\tot A_{kcr}^{*,*} (X,\mathfrak{U};,V)^G$ is cohomologous to the image 
$i_{k,eq}^n  (f)$ of the equivariant $n$-cocycle 
$f\in A_{kcr}^n (X,\mathfrak{U};V)^G$ in 
$\tot A_{kcr}^{*,*} (X,\mathfrak{U};V)^G$.
\end{proposition}

\begin{proof}
  The proof is analogous to that of Proposition \ref{contiscohtocr}. 
\end{proof}

\begin{corollary}
 The map 
$H (i_{k,eq})^{-1} H(j_{eq}) :H_{kc,eq}(X;V) \rightarrow 
H_{kcr,eq}(X,\mathfrak{U};V)$ 
is induced by the inclusion 
$A_{kc}^* (X,\mathfrak{U};V)^G \hookrightarrow A_{kcr}^* (X,\mathfrak{U};V)^G$. 
\end{corollary}

\begin{corollary}
If the morphism 
$j^*_{k, eq}:={j^*}^G : A_{kc}^* (X;V)^G \rightarrow 
\tot A_{kcr}^{*,*}(X,\mathfrak{U}A)^G$ induces a monomorphism, epimorphism or
isomorphism in cohomology, then the inclusion 
$A_{kc}^* (X;V)^G \hookrightarrow A_{kcr}^* (X,\mathfrak{U};V)^G$
induces a monomorphism, epimorphism or isomorphism in cohomology respectively.
\end{corollary}

\begin{lemma}
  For any directed system $\{ \mathfrak{U}_i \mid i \in I \}$ of open 
coverings of $X$ the morphism 
$\colim_i i^*:  \colim_i A_{kcr}^* ( X,\mathfrak{U}_i ; V) \rightarrow 
\tot \colim_i A_{kcr}^{*,*} (X,\mathfrak{U}_i ; V)$ induces an isomorphism in 
cohomology.
\end{lemma}

\begin{proof}
The passage to the colimit preserves the exactness of the augmented row
complexes (Lemma \ref{columnsexactk}). 
\end{proof}

\begin{lemma}
  For any directed system $\{ \mathfrak{U}_i \mid i \in I \}$ of $G$-invariant
  open coverings of $X$ the morphism 
$\colim_i i_{k, eq}^*:  \colim_i A_{kcr}^* ( X,\mathfrak{U}_i ; V)^G \rightarrow 
\tot \colim_i A_{cr}^{*,*} (X,\mathfrak{U}_i ; V)^G$ induces an isomorphism in 
cohomology.
\end{lemma}

\begin{proof}
The passage to the colimit preserves the exactness of the augmented row
complexes (Lemma \ref{columnsexacteqk}). 
\end{proof}

If the open diagonal neighbourhoods  $\fktop \mathfrak{U}[n]$ in 
$\fktop X^{n+1}$ for open 
coverings $\mathfrak{U}$ of $X$ are cofinal in the directed set of all open 
diagonal neighbourhoods then one obtains the double complex complex 
\begin{equation*}
  A_{kcg}^{*,*} (X;V):= \colim_{\mathfrak{U} \text{ is open cover of $X$}} 
A_{kcr}^{*,*} (X;\mathfrak{U};V)
\end{equation*}
whose rows and columns are augmented by the complexes 
$A_{kcg}^* (X;V)$ and $A_{kc}^* (X;V)$ respectively. 
In this case the colimit morphism 
$i_{kcg}^* : A_{kcg}^* ( X; V) \rightarrow \tot A_{kcg}^{*,*} (X; V)$ induces an 
isomorphism in cohomology. Furthermore the colimit double complex 
$A_{kcg}^{*,*} (X;V)$ then is a double complex of $G$-modules and the 
$G$-equivariant cochains in form a sub double complex $A_{kcg}^{*,*} (X;V)^G$, 
whose rows and columns are augmented by the colimit 
complex $A_{kcg,eq}^* (X;V)$ and by the complex $A_{kc}^* (X;V)^G$ respectively.
In addition we observe:

\begin{lemma} \label{natinclofeqdcisisok}
If the open diagonal neighbourhoods $\fktop \mathfrak{U}[n]$ in 
$\fktop X^{n+1}$ for open coverings $\mathfrak{U}$ of $X$ are cofinal in the 
directed set of all open diagonal neighbourhoods then the natural morphism of 
double complexes 
  \begin{equation*}
A_{kcg,eq}^{*,*} (X;V):= \colim_{\mathfrak{U} 
\text{is $G$-invariant open cover of $X$}} A_{kcr}^{*,*} (X;\mathfrak{U};V)^G 
\rightarrow A_{kcg}^* (X;V)^G  
\end{equation*}
is a natural isomorphism.
\end{lemma}

\begin{proof}
  The proof is analogous to that of Proposition \ref{natinclofeqccisisok}.
\end{proof}

As a consequence the colimit morphism 
$i_{kcg,eq}^* : A_{kcg,eq}^* ( X; V) \rightarrow \tot A_{kcg}^{*,*} (X; V)^G$ 
then induces an isomorphism in cohomology, and the morphism 
$H (i_{kcg,eq})$ is invertible. For the composition 
$H(i_{kcg,eq})^{-1} H(j_{k,eq}):H_{kc,eq}(X;V)\rightarrow 
H_{kcg,eq}(X,\mathfrak{U};V)$ we observe:

\begin{proposition} \label{contiscohtocreqk}
If the open diagonal neighbourhoods $\fktop \mathfrak{U}[n]$ in 
$\fktop X^{n+1}$ for open 
coverings $\mathfrak{U}$ of $X$ are cofinal in the directed set of all open 
diagonal neighbourhoods then the image $j^n (f)$ of a continuous equivariant 
$n$-cocycle $f$ on $X$ in $\tot A_{kcg}^{*,*} (X;,V)^G$ is cohomologous to the 
image $i_{kcg,eq}^n  (f)$ of the equivariant cocycle 
$f\in A_{kcg,eq}^n (X;V)$ in $\tot A_{kcg}^{*,*} (X;V)^G$.
\end{proposition}

\begin{proof}
  The proof is analogous to that of Proposition \ref{contiscohtocr}.
\end{proof}

\begin{corollary}
If the open diagonal neighbourhoods $\fktop \mathfrak{U}[n]$ in 
$\fktop X^{n+1}$ for open 
coverings $\mathfrak{U}$ of $X$ are cofinal in the directed set of all open 
diagonal neighbourhoods then the composition 
$H (i_{kcg,eq})^{-1} H(j_{k, eq}):H_{kc,eq}(X;V)\rightarrow H_{kcg,eq}(X;V)$ 
is induced by the inclusion 
$A_{kc}^* (X;V)^G \hookrightarrow A_{kcg}^* (X;V)^G$. 
\end{corollary}

\begin{corollary}
If the open diagonal neighbourhoods $\fktop \mathfrak{U}[n]$ in 
$\fktop X^{n+1}$ for open coverings $\mathfrak{U}$ of $X$ are cofinal in the 
directed set of all open diagonal neighbourhoods and the morphism 
$j^*_{k, eq}:={j_k^*}^G : A_{kc}^* (X;V)^G \rightarrow \tot A_{kcg}^{*,*}(X;V)^G$
induces a monomorphism, epimorphism or isomorphism in cohomology, then the 
inclusion $A_{kc}^* (X;V)^G \hookrightarrow A_{kcg,eq}^* (X;V)$ induces a 
monomorphism, epimorphism or isomorphism in cohomology respectively.
\end{corollary}

\section{Continuous and $\mathfrak{U}$ -Continuous Cochains on $k$-spaces}
\label{seccontanduccck}

In this section we consider transformation $k$-groups $(G,X)$ and $G$-modules 
$V$ in $\ktop$ for which we show that the inclusion 
$A_{kc}^* (X,\mathfrak{U};V)^G \hookrightarrow A_{kcr}^* (X,\mathfrak{U};V)^G$
of the complex of continuous equivariant cochains into the complex 
of equivariant $\mathfrak{U}$-continuous cochains induces an isomorphism 
$H_{kc}^* (X,\mathfrak{U};V) \cong H_{kcr}^* (X,\mathfrak{U};V)$ provided the
$k$-space $X$ is contractible. The proceeding is similar to that in
Section \ref{seccontanduccck}. At first we reduce the problem to the 
non-equivariant case:

\begin{proposition} \label{noneqextheneqexk}
 If the augmented column complexes 
$A_{kc}^p (X;V) \hookrightarrow A_{kcr}^{p,*}(X,\mathfrak{U};V)$ 
are exact, then the augmented sub column complexes
$A_{kc}^p (X;V)^G \hookrightarrow A_{kcr}^{p,*}(X,\mathfrak{U};V)^G$  
of equivariant cochains are exact as well.
\end{proposition}

\begin{proof}
 The proof is analogous to that of Proposition \ref{noneqextheneqex}.
\end{proof}

\begin{corollary} \label{augexthenjeqindisok}
If the augmented column complexes 
$A_{kc}^p (X;V) \hookrightarrow A_{kcr}^{p,*} (X,\mathfrak{U};V)$ 
are exact, then the inclusion 
$j_{k, eq}^* : A_{kc}^* (X;V)^G \hookrightarrow 
\tot A_{kcr}^{*,*}(X,\mathfrak{U},V)^G$ induces an isomorphism in cohomology.
\end{corollary}

\begin{corollary} \label{augextheninclindisok}
If the augmented column complexes 
 $A_{kc}^p (X;V) \hookrightarrow A_{kcr}^{p,*} (X,\mathfrak{U};V)$ 
are exact, then the inclusion 
$A_{kc}^* (X;V)^G \hookrightarrow A_{kcr}^* (X,\mathfrak{U};V)^G$
induces an isomorphism in cohomology.
\end{corollary}

To achieve the announced result it remains to show that for contractible 
$k$-spaces $X$ the colimit augmented columns 
$A_{kc}^p (X;V) \hookrightarrow A_{kcg}^{p,*} (X;V)$ are exact. For this purpose
we first consider the cochain complex associated to the cosimplicial abelian 
group 
\begin{equation*}
  A_k^{p,*} (X;V) := \left\{ f : X^{p+1} \times X^{*+1} \rightarrow V 
\mid \, \forall \vec{x}' \in X^{*+1} :f (-,\vec{x}') \in C ( \fktop X^{p+1},V)
\right\}
\end{equation*}
of global cochains, its subcomplex $A_{kcr}^{p,*} (X,\mathfrak{U};V)$ and 
the cochain complexes associated to the cosimplicial abelian groups 
\begin{eqnarray*}
 A_k^{p,*} (\mathfrak{U};V) & := & 
\{ f : X^{p+1} \times \mathfrak{U}[*] \rightarrow 
\mid \, \forall \vec{x}' \in \mathfrak{U}[*] : f (-,\vec{x}') \in 
C (\fktop X^{p+1},V) \} \quad \text{and} \\ 
A_{kc}^{p,*} (X,\mathfrak{U};V) & := & 
C ( \fktop X^{p+1} \times_k \fktop \mathfrak{U}[*] , V) \, .
\end{eqnarray*}

Restriction of global to local cochains induces morphisms of cochain complexes 
$\res^{p,*}_k : A_k^{p,*} (X;V) \twoheadrightarrow  A_k^{p,*} (X,\mathfrak{U};V)$ and 
$\res_{kcr}^{p,*} : A_{kcr}^* (X,\mathfrak{U};V) \twoheadrightarrow  
A_{kc}^{p,*} (X,\mathfrak{U};V)$ 
intertwining the inclusions of the subcomplexes 
$A_{kcr}^{p,*} (X,\mathfrak{U};V) \hookrightarrow A_k^{p,*} (X;V)$ and 
$A_{kc}^{p,*} (X,\mathfrak{U};V) \hookrightarrow A_k^{p,*}
(X,\mathfrak{U};V)$, 
so one obtains the following commutative diagram
\begin{equation} \label{morphexseqk}
\begin{array}{cccccccc}
 0 \longrightarrow & \ker (\res_{kcr}^{p,*} ) & \longrightarrow & 
A_{kcr}^{p,*} (X,\mathfrak{U};V) 
& \longrightarrow & A_{kc}^{p,*} (X,\mathfrak{U};V) & \longrightarrow 0 \\ 
& \downarrow & & \downarrow & & \downarrow & \\
0 \longrightarrow & \ker (\res^{p,*}_k ) & \longrightarrow & A_k^{p,*} (X;V) 
& \longrightarrow & A_k^{p,*} (X , \mathfrak{U};V) & \longrightarrow 0
\end{array}
\end{equation}
of cochain complexes whose rows are exact. The kernel $\ker (\res^{p,q}_k )$ is
the subspace of those $(p,q)$-cochains which are trivial on 
$\fktop X^{p+1} \times_k \fktop \mathfrak{U} [q]$. Since these
$(p,q)$-cochains are continuous on 
$\fktop X^{p+1} \times_k \fktop \mathfrak{U}[q]$ 
we find that both kernels coincide. We abbreviate the complex 
$\ker (\res^{p,*}_k ) = \ker (\res_{krc}^{p,*} )$ by $K_k^{p,*}$ and denote the 
cohomology groups of the complex $A_{kcr}^{p,*} (X,\mathfrak{U};V)$ by 
$H_{kcr}^{p,*} (X,\mathfrak{U};V)$, the cohomology groups of the complex 
$A_{kc}^{p,*} (X,\mathfrak{U};V)$ of continuous cochains by 
$H_{kc}^{p,*} (X,\mathfrak{U};V)$ and the cohomology groups of the complex 
$A_k^{p,*} (X,\mathfrak{U};V)$ by $H_k^{p,*} (X,\mathfrak{U};V)$. 

\begin{lemma}
  The cochain complexes $A_k^{p,*} (X;V)$ are exact.
\end{lemma}

\begin{proof}
For any point $* \in X$ the homomorphisms 
$h^{p,q} : A_k^{p,q} (X;V) \rightarrow A_k^{p,q-1} (X;V)$ given by 
$h^{p,q} (f) (\vec{x},\vec{x}'):=f (\vec{x},*,\vec{x}')$ form a contraction of
the complex $A_k^{p,*} (X;V)$. 
\end{proof}

The morphism of short exact sequences of cochain complexes in Diagram 
\ref{morphexseqk}
gives rise to a morphism of long exact cohomology sequences, in which the 
cohomology of the complex $A_k^{p,*} (X;V)$ is trivial:
\begin{equation} \label{diaglecsk}
\xymatrix@R-10pt@C-4pt{ 
\ar[r] & H^q (K_k^{p,*} )) \ar[r] \ar@{=}[d] 
& H_{kcr}^{p,q} (X,\mathfrak{U};V) \ar[r] \ar[d]  
& H_{kc}^{p,q} (X,\mathfrak{U};V) \ar[r] \ar[d]  
& H^{q+1}  (K_k^{p,*} ) \ar[r] \ar@{=}[d] & {} \\ 
\ar[r]^\cong & H^q (K^{p,*} ) \ar[r]&  0  \ar[r] 
& H_k^{p,q} (X,\mathfrak{U};V) \ar[r]^\cong &  H^{q+1}  (K^{p,*} )) \ar[r] & {}
}
\end{equation}

\begin{lemma}
The augmented complex 
$A_{kc}^p (X;V) \hookrightarrow A_{cr}^{p,*} (X,\mathfrak{U};V)$ is exact if and
only if the inclusion 
$A_{kc}^{p,*} (X,\mathfrak{U};V) \hookrightarrow A_k^{p,*} (X,\mathfrak{U};V)$ 
induces an isomorphism in cohomology.
\end{lemma}

\begin{proof}
 This is an immediate consequence of Diagram \ref{diaglecsk}
\end{proof}

\begin{proposition}
  If the inclusion 
$A_{kc}^{p,*} (X,\mathfrak{U};V) \hookrightarrow A_k^{p,*} (X,\mathfrak{U};V)$ 
induces an isomorphism in cohomology, then the inclusions  
$A_c^* (X;V)^G \hookrightarrow 
\tot A_{kcr}^{*,*}(X,\mathfrak{U},V)^G$ and 
$A_{kc}^* (X,\mathfrak{U};V)^G \hookrightarrow A_{kcr}^* (X,\mathfrak{U};V)^G$
also induces an isomorphism in cohomology.
\end{proposition}

\begin{proof}
  This follows from the preceding Lemma and Corollaries 
\ref{augexthenjeqindisok} and \ref{augextheninclindisok}.
\end{proof}

For $k$-spaces $X$ for which the open diagonal neighbourhoods 
$\fktop \mathfrak{U}[n]$ in $\fktop X^{n+1}$ for open coverings 
$\mathfrak{U}$ of $X$ are 
cofinal in the directed set of all open diagonal neighbourhoods
the passage to the colimit over all open coverings of $X$ yields the
corresponding results for the complexes of cochains with continuous germs: 

\begin{proposition} \label{noneqextheneqexcgk}
 If the open diagonal neighbourhoods $\fktop \mathfrak{U}[n]$ in 
$\fktop X^{n+1}$ for open coverings $\mathfrak{U}$ of $X$ are cofinal in the 
directed set of all open diagonal neighbourhoods and the augmented column 
complexes 
$A_{kc}^p (X;V) \hookrightarrow A_{kcg}^{p,*}(X;V)$ 
are exact, then the augmented sub column complexes
$A_{kc}^p (X;V)^G \hookrightarrow A_{kcg}^{p,*}(X;V)^G$  
of equivariant cochains are exact as well.
\end{proposition}

\begin{proof}
  The proof is similar to that of Proposition \ref{noneqextheneqex}.
\end{proof}

\begin{corollary} \label{augexthenjeqindisocgk}
If the open diagonal neighbourhoods $\fktop \mathfrak{U}[n]$ in 
$\fktop X^{n+1}$ for open 
coverings $\mathfrak{U}$ of $X$ are cofinal in the directed set of all open 
diagonal neighbourhoods and the augmented column complexes 
$A_{kc}^p (X;V) \hookrightarrow A_{kcg}^{p,*} (X;V)$ 
are exact, then the inclusion 
$j_{k, eq}^* : A_{kc}^* (X;V)^G \hookrightarrow 
\tot A_{kcg}^{*,*}(X;V)^G$ induces an isomorphism in cohomology.
\end{corollary}

\begin{corollary} \label{augextheninclindisocgk}
If the open diagonal neighbourhoods $\fktop \mathfrak{U}[n]$ in 
$\fktop X^{n+1}$ for open 
coverings $\mathfrak{U}$ of $X$ are cofinal in the directed set of all open 
diagonal neighbourhoods and the augmented column complexes 
 $A_{kc}^p (X;V) \hookrightarrow A_{kcg}^{p,*} (X;V)$ 
are exact, then the inclusion 
$A_{kc}^* (X;V)^G \hookrightarrow A_{kcg}^* (X;V)^G$
induces an isomorphism in cohomology.
\end{corollary}

\begin{remark} \label{remonlyginvcovk}
Alternatively to taking the colimit over all open coverings $\mathfrak{U}$ of 
$X$ one may consider $G$-invariant open coverings only to obtains the same
results. 
(This was shown in Proposition \ref{natinclofeqccisisok} and Lemmata 
\ref{natinclofeqdcisisok}.)
\end{remark}

\begin{example}
  If $G=X$ is a metrisable, locally compact or Hausdorff $k_\omega$
  topological group which acts on itself by left translation and the augmented 
columns  
$A_{kc}^p (X;V) \hookrightarrow A_{kcg}^{p,*} (X;V) := 
\colim A_k^{p,*} (X,\mathfrak{U}_U;V)$ (where $U$ runs over all open identity
neighbourhoods in $G$)  
are exact, then $A_{kc}^* (X;V)^G \hookrightarrow A_{kcg}^* (X;V)^G$
induces an isomorphism in cohomology.
\end{example}

The complex $A^{p,*} (X,\mathfrak{U};V)$ is isomorphic to the complex 
$A^* (\mathfrak{U}; C (\fktop X^{p+1},V))$. 
If the open diagonal neighbourhoods  $\mathfrak{U}[n]$ in $X^{n+1}$ for open 
coverings $\mathfrak{U}$ of $X$ are cofinal in the directed set of all open 
diagonal neighbourhoods then the colimit 
$A_{AS}^* (X; C (X^{p+1} ,V)):=\colim A^* (\mathfrak{U}; C (X^{p+1} ,V))$, 
where $\mathfrak{U}$ runs over all open coverings of
$X$ is the complex of Alexander-Spanier cochains on $X$ (with values in 
$C (X^{p+1} ,V)$). In this case the colimit complex 
$\colim A^p (X; A^* (\mathfrak{U};V))$ is isomorphic to the cochain complex 
$A_{AS}^* (X;C (\fktop X^{p+1} ,V))$. 
A similar observation can be made for the cochain complex 
$A_{kc}^{p,*} (X,\mathfrak{U};V)$ because the exponential law 
$C ( \fktop X^{p+1} \times_k \fktop \mathfrak{U}[q],V) \cong 
C (X, \fktop C (\mathfrak{U}[q],V))$ holds in $\ktop$.
Passing to the colimit in Diagram \ref{morphexseqk} yields the morphism

\begin{equation} \label{morphexseqcgk}
\begin{array}{cccccccc}
 0 \longrightarrow & \ker (\res_{kcg}^{p,*} ) & \longrightarrow & 
A_{kcg}^{p,*} (X;V) & \longrightarrow & \colim A_{kc}^{p,*} (X,\mathfrak{U};V) & 
\longrightarrow 0 \\ 
& \downarrow & & \downarrow & & \downarrow & \\
0 \longrightarrow & \ker (\res_k^{p,*} ) & \longrightarrow & A_k^{p,*} (X;V) 
& \longrightarrow & A_{AS}^* (X; C^{p+1} (X,V))
& \longrightarrow 0
\end{array}
\end{equation}
of short exact sequences of cochain complexes. The kernel $\ker (\res^{p,q}_k)$
is the subspace of those $(p,q)$-cochains which are trivial on 
$\fktop X^{p+1} \times_k \fktop \mathfrak{U} [q]$ for some open covering 
$\mathfrak{U}$ of $X$. Since these $(p,q)$-cochains are continuous on 
$\fktop X^{p+1} \times_k \fktop \mathfrak{U}[q]$ we find that both kernels 
coincide. We abbreviate the complex 
$\ker (\res^{p,*}_k ) = \ker (\res_{kcg}^{p,*} )$ by 
$K_{kcg}^{p,*}$ and denote the cohomology groups of the complex 
$A_{kcg}^{p,*} (X;V)$ by $H_{kcg}^{p,*} (X;V)$. 
The morphism of short exact sequences of cochain complexes in Diagram 
\ref{morphexseqcgk} gives rise to a morphism of long exact cohomology 
sequences:
\begin{equation} \label{diaglecscgk}
\xymatrix@C-11pt@R-10pt{ 
\ar[r] & H^q (K_{kcg}^{p,*} )) \ar[r] \ar@{=}[d] 
& H_{kcg}^{p,q} (X,\mathfrak{U};V) \ar[r] \ar[d]  
& H^q (\colim A_{kc}^{p,*} (X,\mathfrak{U};V) \ar[r] \ar[d]  
& H^{q+1}  (K_{kcg}^{p,*} ) \ar[r] \ar@{=}[d] & {} \\ 
\ar[r]^\cong & H^q (K_{kcg}^{p,*} ) \ar[r]&  0  \ar[r] 
& H_{AS}^q (X; C^{p+1} (X,V)) \ar[r]^\cong &  H^{q+1}  (K_{kcg}^{p,*} )) \ar[r] 
& {}
}
\end{equation}

\begin{lemma}
The augmented complex 
$A_{kc}^p (X;V) \hookrightarrow A_{kcg}^{p,*} (X;V)$ is exact if and only if 
the inclusion 
$\colim A_{kc}^{p,*} (X,\mathfrak{U};V)\hookrightarrow A_{AS}^* (X;C^{p+1}(X,V))$
of cochain complexes induces an isomorphism in cohomology.
\end{lemma}

\begin{proof}
 This is an immediate consequence of Diagram \ref{diaglecscgk}
\end{proof}

\begin{proposition} \label{inclcolimacpsindisothenjeqisok}
If the open diagonal neighbourhoods  $\fktop \mathfrak{U}[n]$ in 
$\fktop X^{n+1}$ for open coverings $\mathfrak{U}$ of $X$ are cofinal in the 
directed set of all open diagonal neighbourhoods and the inclusion 
$\colim A_{kc}^{p,*}(X,\mathfrak{U};V)\hookrightarrow A_{AS}^* (X;C(X^{p+1},V))$ 
induces an isomorphism in cohomology, then  
$j_{k, eq}^* : A_{kc}^* (X;V)^G \hookrightarrow \tot A_{kcg}^{*,*}(X;V)^G$ and 
$A_{kc}^* (X;V)^G \hookrightarrow A_{kcg}^* (X;V)^G$
also induce an isomorphism in cohomology.
\end{proposition}

\begin{proof}
  This follows from the preceding Lemma and Corollaries 
\ref{augexthenjeqindisocgk} and \ref{augextheninclindisocgk}.
\end{proof}

As observed before (cf. Remark \ref{remonlyginvcovk}) one may restrict oneself
to the directed system of $G$-invariant open coverings only to achieve the
same result. Thus we observe:

\begin{corollary}
If $G=X$ is a locally contractible metrisable, locally contractible locally
compact or locally contractible Hausdorff $k_\omega$ topological group which 
acts on itself by left translation and the inclusion 
$\colim A_{kc}^{p,*} (X,\mathfrak{U};V)\hookrightarrow A_{AS}^* (X;C(X^{p+1},V))$ 
(where $U$ runs over all open identity
neighbourhoods in $G$)  induces an isomorphism in cohomology, then the 
inclusion $A_{kc}^* (X;V)^G \hookrightarrow A_{kcg}^* (X;V)^G$
induces an isomorphism in cohomology as well.
\end{corollary}

\begin{proof}
 It has been shown in \cite{vE62b} that the cohomology of the colimit cochain 
complex 
$\colim A^* ( \mathfrak{U} ;C(\fktop X^{p+1},V))$ is the Alexander-Spanier 
cohomology of $X$ with coefficients $C(\fktop X^{p+1},V)$. 
\end{proof}

\begin{lemma} \label{xcontrthenacpsistrivk}
  If the topological space $X$ is contractible, then the cohomology 
of the complex $\colim A_{kc}^{p,*} (X,\mathfrak{U};V)$ is trivial.
\end{lemma}

\begin{proof}
The reasoning is analogous to that for the Alexander-Spanier presheaf. 
The proof \cite[Theorem 2.5.2]{F10} carries over almost in verbatim. 
\end{proof}

\begin{theorem}
For contractible $X$ the inclusion 
$A_{kc}^* (X;V)^G \hookrightarrow A_{kcg}^* (X;V)^G$ 
induces an isomorphism in cohomology.
\end{theorem}

\begin{proof}
If the $k$-space $X$ is contractible, then the Alexander-Spanier
cohomology of $X$ is trivial and the cohomology of the cochain 
complex $\colim A_{kc}^{p,*} (X,\mathfrak{U};V)$ is trivial by Lemma 
\ref{xcontrthenacpsistrivk}. By Proposition \ref{inclcolimacpsindisothenjeqisok}
the inclusion $A_{kc}^* (X;V)^G \hookrightarrow A_{kcg}^* (X;V)^G$ then
induces an isomorphism in cohomology.
\end{proof}

\begin{corollary}
For metrisable, locally compact or Hausdorff $k_\omega$ topological groups $G$
which are contractible the continuous group cohomology 
$H_{kc,eq} (G;V)$ is isomorphic to the cohomology 
$H_{kcg,eq} (G;V)$ of homogeneous group 
cochains with continuous germ at the diagonal.
\end{corollary}

\section{Complexes of Smooth Cochains}

In this Section we introduce the sub (double)complexes for smooth 
transformation groups $(G,M)$ and smooth $G$-modules $V$, where $V$ is an 
abelian Lie group. 
(We use the general infinite dimensional calculus introduced in \cite{BGN04}.)
Let $(G,M)$ be a smooth transformation group, $V$ be a smooth $G$-module and 
$\mathfrak{U}$ be an open covering of $M$. 

\begin{definition}
For every manifold $M$ and abelian Lie group $V$ the 
subcomplex $A_s^* (M;V):= C^\infty (M^{*+1};V)$ of the standard complex is 
called the \emph{smooth standard complex}. 
The cohomology $H_{eq,s} (M;V)$ of the subcomplex $A_s^* (M;V)^G$ is called 
the equivariant smooth cohomology of $M$ (with values in $V$).
\end{definition}

\begin{example}
  For any Lie group $G$ which acts on itself by left translation and
  smooth $G$-module $V$ the complex $A_s^* (G;V)^G$ is the complex of smooth 
(homogeneous) group cochains; its cohomology $H_{eq,s} (G;V)$ is the 
smooth group cohomology of $G$ with values in $V$. 
\end{example}
 
For Lie groups $G$ and $G$-modules $V$ the first cohomology group 
$H_{eq,s}^1 (G;V)$ classifies smooth crossed morphisms modulo principal 
derivations, the second cohomology group $H_{eq,s}^2 (G;V)$ classifies 
equivalence classes of Lie group extensions 
$V \hookrightarrow \hat{G} \twoheadrightarrow G$ which admit a smooth global 
section (i.e. $\hat{G} \twoheadrightarrow G$ is a  trivial smooth
$V$-principal bundle) and the third cohomology group $H_{eq,c}^3 (G;V)$
classifies equivalence classes of smoothly split crossed modules.

For each open covering $\mathfrak{U}$ of $M$ one can consider the subcomplex 
of $A^* (M;V)$ formed by the groups  
\begin{equation*}
A_{sr}^n (M,\mathfrak{U};V) := 
\left\{ f \in  A^n (M;V) \mid \, f_{\mid \mathfrak{U} [n]}
  \in C^\infty (\mathfrak{U}[n];V) \right\}  
\end{equation*}
of cochains whose restriction to the subspaces $\mathfrak{U}[n]$ of $M^{n+1}$ 
are smooth. The cohomology of the cochain complex 
$A_{sr}^* (M,\mathfrak{U};V)$ is denoted by $H_{sr} (M,\mathfrak{U};V)$. 
If the covering $\mathfrak{U}$ of $M$ is $G$-invariant, then the subspaces 
$\mathfrak{U}[*]$ is a simplicial $G$-subspace of the simplicial $G$-space 
$M^{*+1}$. 
For $G$-invariant coverings $\mathfrak{U}$ of $M$ the cohomology of the 
subcomplex $A_{sr}^* (M,\mathfrak{U};V)^G$ of $G$-equivariant cochains is 
denoted by $H_{cr,eq} (M,\mathfrak{U};V)$. 

\begin{example}
  If $G=M$ is a Lie group which acts on itself by left translation and
  $U$ an open identity neighbourhood, then the complex 
$A_{sr}^* (M,\mathfrak{U}_U;V)^G$ is the complex of homogeneous group cochains 
whose restrictions to the subspaces $\mathfrak{U}_U [*]$ are smooth. 
(These are sometimes called $\mathfrak{U}$-smooth cochains.) 
\end{example}

For directed systems $\{ \mathfrak{U}_i \mid i \in I \}$ of open coverings of
$M$ one can also consider the colimit complex 
$\colim_i A_{sr}^* (M,\mathfrak{U}_i ;V)$. In particular for the directed 
system of all open coverings of $M$ one observes that the open diagonal 
neighbourhoods $\mathfrak{U}[n]$ in $M^{n+1}$ for open coverings 
$\mathfrak{U}$ of $M$ are cofinal in the directed set of all open diagonal 
neighbourhoods, hence one obtains the complex
\begin{equation*}
  A_{sg}^* (M;V):= \colim_{\mathfrak{U} \text{is open cover of $M$}} A_{sr}^*
  (M;\mathfrak{U};V)
\end{equation*}
of global cochains whose germs at the diagonal are continuous. This is a
subcomplex of the standard complex $A^* (M;V)$ which is invariant under the 
$G$-action (Eq. \ref{defgact}) and thus a sub complex of $G$-modules. 
The $G$-equivariant cochains with continuous germ form a subcomplex 
$A_{sg}^* (M;V)^G$ thereof, whose cohomology is denoted by $H_{cg,eq} (M;V)$. 
The latter subcomplex can also be obtained by taking the colimit over all 
$G$-invariant open coverings of $M$ only: 

\begin{proposition} \label{natinclofeqccisisosmooth}
  The natural morphism of cochain complexes 
  \begin{equation*}
A_{cg,eq}^* (M;V):= \colim_{\mathfrak{U} \text{is $G$-invariant open cover
      of $M$}} A_{sr}^* (M;\mathfrak{U};V)^G \rightarrow A_{sg}^* (M;V)^G    
  \end{equation*}
is a natural isomorphism.
\end{proposition}

\begin{proof}
  The proof is analogous to that of Proposition \ref{natinclofeqccisiso}.
\end{proof}

\begin{corollary}
  The cohomology $H_{cg,eq} (M;V)$ is the cohomology of the complex of 
equivariant cochains which are continuous on some $G$-invariant neighbourhood
of the diagonal.
\end{corollary}

\begin{example}
  If $G=M$ is a Lie group which acts on itself by left translation, 
then the complex $A_{sg}^* (G;V)^G$ is the complex of homogeneous group
cochains whose germs at the diagonal are smooth. 
(By abuse of language these are sometimes called 'locally smooth' 
group cochains.) 
\end{example}

We will show (in Section \ref{secsmoothanduscc}) that the
inclusion $A_{sr}^* (M,\mathfrak{U};V) \hookrightarrow A_s^* (M;V)$ induces an 
isomorphism in cohomology provided the manifold $M$ is smoothly contractible. 
For this purpose we consider the abelian groups 
\begin{equation} \label{defrsu}
  A_{sr}^{p,q} ( M,\mathfrak{U} ; V ) := 
\left\{ f: M^{p+1} \times M^{q+1} \rightarrow V 
\mid f_{\mid M^{p+1} \times
      \mathfrak{U}[q]} \; \text{is continuous} \right\}
\, .
\end{equation}
The abelian groups $A_{sr}^{p,q} ( M,\mathfrak{U} ; V )$ 
form a first quadrant sub double complex of the double complex 
$A_{cr}^{p,q} ( M,\mathfrak{U} ; V )$.
The rows of the double complex $A_{sr}^{*,*} ( M,\mathfrak{U} ; V )$ can be
augmented by the complex $A_{sr}^* (M,\mathfrak{U} ;V)$ for the covering
$\mathfrak{U}$ and the columns can be augmented by the exact complex 
$A_s^* (M;V)$ of continuous cochains:
\begin{equation*}
  \vcenter{
  \xymatrix{
\vdots & \vdots & \vdots & \vdots \\ 
A_{sr}^2 (M, \mathfrak{U};V) \ar[r] \ar[u]_{d_{v}} 
& A_{sr}^{0,2} ( M, \mathfrak{U} ; V) \ar[r]^{d_{h}}\ar[u]_{d_{v}} 
&  A_{sr}^{1,2} ( M, \mathfrak{U}; V) \ar[r]^{d_{h}}\ar[u]_{d_{v}} 
& A_{sr}^{2,2} ( M, \mathfrak{U} ; V) \ar[r]^{d_{h}}\ar[u]_{d_{v}} & \cdots \\
A_{sr}^1 (M, \mathfrak{U};V) \ar[r] \ar[u]_{d_{v}} 
& A_{sr}^{0,1} ( M, \mathfrak{U} ; V) \ar[r]^{d_{h}}\ar[u]_{d_{v}} 
&  A_{sr}^{1,1} ( M, \mathfrak{U} ; V) \ar[r]^{d_{h}}\ar[u]_{d_{v}} 
& A_{sr}^{2,1} ( M, \mathfrak{U} ; V) \ar[r]^{d_{h}}\ar[u]_{d_{v}} & \cdots \\
A_{sr}^0 (M, \mathfrak{U};V) \ar[r] \ar[u]_{d_{v}} 
& A_{sr}^{0,0} ( M, \mathfrak{U} ; V) \ar[r]^{d_{h}}\ar[u]_{d_{v}} 
&  A_{sr}^{1,0} ( M, \mathfrak{U} ; V) \ar[r]^{d_{h}}\ar[u]_{d_{v}} 
&  A_{sr}^{2,0} ( M, \mathfrak{U} ; V) \ar[r]^{d_{h}}\ar[u]_{d_{v}} & \cdots \\
& A_s^0 ( M ; V) \ar[r]^{d_{h}}\ar[u] 
& A_s^1 ( M ; V) \ar[r]^{d_{h}}\ar[u] 
& A_s^2 ( M ; V) \ar[r]^{d_{h}}\ar[u] 
& \cdots
}} 
\end{equation*}

We denote the total complex of the double complex 
$A_{sr}^{*,*} ( M, \mathfrak{U} ; V)$ by 
$\tot  A_{sr}^{*,*} ( M, \mathfrak{U} ; V)$. 
The augmentations of the rows and columns of this double complex 
induce morphisms $i^* : A_{sr}^* ( M, \mathfrak{U} ; V) \rightarrow 
\tot A_{sr}^{*,*} ( M, \mathfrak{U} ; V)$ and 
$j^*:  A_s^* ( M ; V) \rightarrow \tot A_{sr}^{*,*} ( M,\mathfrak{U} ; V)$ 
of cochain complexes respectively. 

\begin{lemma} \label{columnsexactsmooth}
  The morphism $i^*:  A_{sr}^* ( M,\mathfrak{U} ; V) \rightarrow 
\tot A_{sr}^{*,*} (M,\mathfrak{U} ; V)$ induces an isomorphism in cohomology.
\end{lemma}

\begin{proof}
  The row contraction given in the proof of Lemma \ref{columnsexact} 
restricts to one of the sub row complex  
$A_s^* (M,\mathfrak{U} ; V) \hookrightarrow A_{sr}^{*,*} (M,\mathfrak{U} ; V)$.
\end{proof}

\begin{remark}
  Note that this construction does not work for the column complexes. 
\end{remark}

For $G$-invariant open coverings $\mathfrak{U}$ of $M$ one can consider the 
sub double complex $A_{sr}^{*,*} ( M,\mathfrak{U} ; V )^G$ of 
$A_{sr}^{*,*} ( M,\mathfrak{U} ; V )$ whose rows are augmented by the cochain 
complex $A_{sr}^* (M,\mathfrak{U} ;V)^G$ for the covering $\mathfrak{U}$ and 
the columns can be augmented by the complex $A_s^* (M;V)^G$ of smooth 
equivariant cochains (,which is not exact in general). 

\begin{lemma} \label{columnsexacteqsmooth}
  For $G$-invariant coverings $\mathfrak{U}$ of $M$ the morphism 
$i_{eq}^*:  ={i^*}^G$ induces an isomorphism in cohomology.
\end{lemma}

\begin{proof}
  The contraction $h_{*,q}$ of the augmented rows 
$A_{cr}^q ( M, \mathfrak{U} ; V) \hookrightarrow 
\tot A_{sr}^{*,q} ( M, \mathfrak{U} ; V)$ defined in Eq. \ref{defrowcontr} is
$G$-equivariant and thus restricts to a row contraction of the augmented
sub-row $A_{sr}^q ( M, \mathfrak{U} ; V)^G \hookrightarrow 
\tot A_{sr}^{*,q} ( M, \mathfrak{U} ; V)^G$.
\end{proof}

So the morphism 
$H (i_{eq}) : H_{cr,eq} (M,\mathfrak{U};V) \rightarrow 
H ( \tot A_{sr}^{*,*} ( M, \mathfrak{U} ; V)^G )$ is invertible. 
For the composition 
$H (i_{eq})^{-1} H(j_{eq}):H_{c,eq}(M;V)\rightarrow H_{cr,eq}(M,\mathfrak{U};V)$ 
we observe:

\begin{proposition} \label{contiscohtocrsmooth}
The image $j^n (f)$ of a smooth equivariant $n$-cocycle $f$ on $M$ in 
$\tot A_{sr}^{*,*} (M,\mathfrak{U};,V)^G$ is cohomologous to the image 
$i_{eq}^n  (f)$ of the equivariant $n$-cocycle 
$f\in A_{sr}^n (M,\mathfrak{U};V)^G$ in $\tot A_{sr}^{*,*} (M,\mathfrak{U};V)^G$.
\end{proposition}

\begin{proof}
  The proof is analogous to that of Proposition \ref{contiscohtocr}. 
\end{proof}

\begin{corollary}
 The composition 
$H (i_{eq})^{-1} H(j_{eq}):H_{s,eq}(M;V)\rightarrow H_{sr,eq}(M,\mathfrak{U};V)$ 
is induced by the inclusion 
$A_s^* (M,\mathfrak{U};V)^G \hookrightarrow A_{sr}^* (M,\mathfrak{U};V)^G$. 
\end{corollary}

\begin{corollary}
If the morphism 
$j^*_{eq}:={j^*}^G : A_s^* (M;V)^G \rightarrow 
\tot A_{sr}^{*,*}(M,\mathfrak{U}A)^G$ induces a monomorphism, epimorphism or
isomorphism in cohomology, then the inclusion 
$A_s^* (M;V)^G \hookrightarrow A_{sr}^* (M,\mathfrak{U};V)^G$
induces a monomorphism, epimorphism or isomorphism in cohomology respectively.
\end{corollary}

For any directed system $\{ \mathfrak{U}_i \mid i \in I \}$ of open coverings 
of $M$ one can also consider the corresponding augmented colimit double
complexes. In particular for the directed system of all open coverings of $M$ 
one obtains the double complex complex 
\begin{equation*}
  A_{sg}^{*,*} (M;V):= \colim_{\mathfrak{U} \text{ is open cover of $M$}} 
A_{sr}^{*,*} (M;\mathfrak{U};V)
\end{equation*}
whose rows and columns are augmented by the colimit complex 
$A_{sg}^* (M;V)$ and by the complex $A_s^* (M;V)$ respectively. 

\begin{lemma}
  For any directed system $\{ \mathfrak{U}_i \mid i \in I \}$ of open 
coverings of $M$ the morphism 
$\colim_i i^*:  \colim_i A_{sr}^* ( M,\mathfrak{U}_i ; V) \rightarrow 
\tot \colim_i A_{sr}^{*,*} (M,\mathfrak{U}_i ; V)$ induces an isomorphism in 
cohomology.
\end{lemma}

\begin{proof}
The passage to the colimit preserves the exactness of the augmented row
complexes (Lemma \ref{columnsexactsmooth}). 
\end{proof}

As a consequence the colimit morphism 
$i_{sg}^* : A_{sg}^* ( M; V) \rightarrow \tot A_{sg}^{*,*} (M; V)$ induces an 
isomorphism in cohomology.
The colimit double complex $A_{sg}^{*,*} (M;V)$ is a double complex of 
$G$-modules and the $G$-equivariant cochains in form a sub double complex 
$A_{sg}^{*,*} (M;V)^G$, whose rows and columns are augmented by the colimit 
complex $A_{cg,eq}^* (M;V)$ and by the complex $A_s^* (M;V)^G$ respectively.

\begin{lemma}
  For any directed system $\{ \mathfrak{U}_i \mid i \in I \}$ of $G$-invariant
  open coverings of $M$ the morphism 
$\colim_i i_{eq}^*:  \colim_i A_{sr}^* ( M,\mathfrak{U}_i ; V)^G \rightarrow 
\tot \colim_i A_{sr}^{*,*} (M,\mathfrak{U}_i ; V)^G$ induces an isomorphism in 
cohomology.
\end{lemma}

\begin{proof}
The passage to the colimit preserves the exactness of the augmented row
complexes (Lemma \ref{columnsexacteqsmooth}). 
\end{proof}

Moreover, since the open diagonal neighbourhoods $\mathfrak{U}[n]$ in 
$X^{n+1}$ for open coverings $\mathfrak{U}$ of $X$ are cofinal in the 
directed set of all open diagonal neighbourhoods, we observe:

\begin{lemma} \label{natinclofeqdcisisosmooth}
  The natural morphism of double complexes 
  \begin{equation*}
A_{cg,eq}^{*,*} (X;V):= \colim_{\mathfrak{U} \text{is $G$-invariant open cover
      of $X$}} A_{sr}^{*,*} (X;\mathfrak{U};V)^G \rightarrow A_{sg}^* (X;V)^G    
  \end{equation*}
is a natural isomorphism.
\end{lemma}

\begin{proof}
  The proof is analogous to that of Proposition \ref{natinclofeqccisiso}.
\end{proof}

As a consequence the colimit morphism 
$i_{cg,eq}^* : A_{cg,eq}^* ( M; V) \rightarrow \tot A_{sg}^{*,*} (M; V)^G$ 
induces an isomorphism in cohomology, and the morphism 
$H (i_{cg,eq})$ is invertible. For the composition 
$H(i_{cg,eq})^{-1} H(j_{eq}):H_{c,eq}(M;V)\rightarrow H_{cg,eq}(M,\mathfrak{U};V)$ 
we observe:

\begin{proposition} \label{smoothiscohtosreq}
The image $j^n (f)$ of a continuous equivariant $n$-cocycle $f$ on $M$ in 
$\tot A_{sg}^{*,*} (M;,V)^G$ is cohomologous to the image 
$i_{cg,eq}^n  (f)$ of the equivariant $n$-cocycle $f\in A_{cg,eq}^n (M;V)$ 
in $\tot A_{sg}^{*,*} (M;V)^G$.
\end{proposition}

\begin{proof}
  The proof is analogous to that of Proposition \ref{contiscohtocr}.
\end{proof}

\begin{corollary}
 The composition 
$H (i_{cg,eq})^{-1} H(j_{eq}):H_{c,eq}(M;V)\rightarrow H_{cg,eq}(M;V)$ 
is induced by the inclusion 
$A_s^* (M;V)^G \hookrightarrow A_{sg}^* (M;V)^G$. 
\end{corollary}

\begin{corollary}
If the morphism 
$j^*_{eq}:={j^*}^G : A_s^* (M;V)^G \rightarrow \tot A_{sg}^{*,*}(M;V)^G$
induces a monomorphism, epimorphism or isomorphism in cohomology, then the 
inclusion $A_s^* (M;V)^G \hookrightarrow A_{cg,eq}^* (M;V)$ induces a 
monomorphism, epimorphism or isomorphism in cohomology respectively.
\end{corollary}

\section{Smooth and $\mathfrak{U}$-Smooth Cochains} 
\label{secsmoothanduscc}

In this Section we derive results for smooth transformation groups 
$(G,M)$ and smooth $G$-modules $V$, which are analogous to those concerning
continuous cochains. 
Let $(G,M)$ be a smooth transformation group, $V$ be a smooth $G$-module and 
$\mathfrak{U}$ be an open covering of $M$. 

\begin{proposition} \label{noneqextheneqexsmooth}
 If the augmented column complexes 
$A_s^p (M;V) \hookrightarrow A_{sr}^{p,*}(M,\mathfrak{U};V)$ 
are exact, then the augmented sub column complexes
$A_s^p (M;V)^G \hookrightarrow A_{sr}^{p,*}(M,\mathfrak{U};V)^G$  
of equivariant cochains are exact as well.
\end{proposition}

\begin{proof}
  The proof is analogous to that of Proposition \ref{noneqextheneqex}.
\end{proof}

\begin{corollary} \label{augexthenjeqindisosmooth}
If the augmented column complexes 
$A_s^p (M;V) \hookrightarrow A_{sr}^{p,*} (M,\mathfrak{U};V)$ 
are exact, then the inclusion 
$j_{eq}^* : A_s^* (M;V)^G \hookrightarrow 
\tot A_{sr}^{*,*}(M,\mathfrak{U},V)^G$ induces an isomorphism in cohomology.
\end{corollary}

\begin{corollary} \label{augextheninclindisosmooth}
If the augmented column complexes 
 $A_s^p (M;V) \hookrightarrow A_{sr}^{p,*} (M,\mathfrak{U};V)$ 
are exact, then the inclusion 
$A_s^* (M;V)^G \hookrightarrow A_{sr}^* (M,\mathfrak{U};V)^G$
induces an isomorphism in cohomology.
\end{corollary}

It remains to show that for smoothly contractible manifolds $M$ the colimit 
augmented columns 
$A_s^p (M;V) \hookrightarrow A_{sg}^{p,*} (M;V)$ are exact. For this purpose
we first consider the cochain complex associated to the cosimplicial abelian 
group 
$A^{p,*} (M;V) := \left\{ f : M^{p+1} \times M^{*+1} \rightarrow V 
\mid \, \forall \vec{x}' \in M^{*+1} : f (-,\vec{m}') \in C^\infty (M^{p+1},V) 
\right\}$ 
of global cochains, its subcomplex $A_{sr}^{p,*} (M,\mathfrak{U};V)$ and 
the cochain complexes associated to the cosimplicial abelian groups 
\begin{eqnarray*}
 A^{p,*} (\mathfrak{U};V) & := & 
\{ f : M^{p+1} \times \mathfrak{U}[*] \rightarrow 
\mid \, \forall \vec{m}' \in \mathfrak{U}[*] : f (-,\vec{m}') \in C^\infty (M^{p+1},V)
\} \quad \text{and} \\
A_s^{p,*} (M,\mathfrak{U};V) & := & C^\infty ( M^{p+1} \times \mathfrak{U}[*] , V) \, .
\end{eqnarray*}

Restriction of global to local cochains induces morphisms of cochain complexes 
$\res^{p,*} : A^{p,*} (M;V) \twoheadrightarrow  A^{p,*} (M,\mathfrak{U};V)$ and 
$\res_{sr}^{p,*} : A_{sr}^* (M,\mathfrak{U};V) \twoheadrightarrow  
A_s^{p,*} (M,\mathfrak{U};V)$ intertwining the inclusions of the subcomplexes 
$A_{sr}^{p,*} (M,\mathfrak{U};V) \hookrightarrow A^{p,*} (M;V)$ and 
$A_s^{p,*} (M,\mathfrak{U};V) \hookrightarrow A^{p,*} (M, \mathfrak{U};V)$, so
one obtains the following commutative diagram
\begin{equation} \label{morphexseqsmooth}
\begin{array}{cccccccc}
 0 \longrightarrow & \ker (\res_{sr}^{p,*} ) & \longrightarrow & 
A_{sr}^{p,*} (M,\mathfrak{U};V) 
& \longrightarrow & A_s^{p,*} (M,\mathfrak{U};V) & \longrightarrow 0 \\ 
& \downarrow & & \downarrow & & \downarrow & \\
0 \longrightarrow & \ker (\res^{p,*} ) & \longrightarrow & A^{p,*} (M;V) 
& \longrightarrow & A^{p,*} (M , \mathfrak{U};V) & \longrightarrow 0
\end{array}
\end{equation}

of cochain complexes whose rows are exact. The kernel $\ker (\res^{p,q} )$ is
the subspace of those $(p,q)$-cochains which are trivial on 
$M^{p+1} \times \mathfrak{U} [q]$. Since these $(p,q)$-cochains are smooth 
on $M^{p+1} \times \mathfrak{U}[q]$ we find that both kernels coincide. 
We abbreviate the complex 
$\ker (\res^{p,*} ) = \ker (\res_{rc}^{p,*} )$ by $K^{p,*}$ and denote the 
cohomology groups of the complex $A_{sr}^{p,*} (M,\mathfrak{U};V)$ by 
$H_{sr}^{p,*} (M,\mathfrak{U};V)$, the cohomology groups of the complex 
$A_s^{p,*} (M,\mathfrak{U};V)$ of continuous cochains by 
$H_s^{p,*} (M,\mathfrak{U};V)$ and the cohomology groups of the complex 
$A^{p,*} (M,\mathfrak{U};V)$ by $H^{p,*} (M,\mathfrak{U};V)$. 

\begin{lemma}
  The cochain complexes $A^{p,*} (M;V)$ are exact.
\end{lemma}

\begin{proof}
For any point $* \in M$ the homomorphisms 
$h^{p,q} : A^{p,q} (M;V) \rightarrow A^{p,q-1} (M;V)$ given by 
$h^{p,q} (f) (\vec{x},\vec{x}'):=f (\vec{x},*,\vec{x}')$ form a contraction of
the complex $A^{p,*} (M;V)$. 
\end{proof}

The morphism of short exact sequences of cochain complexes in diagram 
\ref{morphexseq}
gives rise to a morphism of long exact cohomology sequences, in which the 
cohomology of the complex $A^{p,*} (M;V)$ is trivial:
\begin{equation} \label{diaglecssmooth}
\xymatrix@R-10pt@C-4pt{ 
\ar[r] & H^q (K^{p,*} )) \ar[r] \ar@{=}[d] 
& H_{sr}^{p,q} (M,\mathfrak{U};V) \ar[r] \ar[d]  
& H_s^{p,q} (M,\mathfrak{U};V) \ar[r] \ar[d]  
& H^{q+1}  (K^{p,*} ) \ar[r] \ar@{=}[d] & {} \\ 
\ar[r]^\cong & H^q (K^{p,*} ) \ar[r]&  0  \ar[r] 
& H^{p,q} (M,\mathfrak{U};V) \ar[r]^\cong &  H^{q+1}  (K^{p,*} )) \ar[r] & {}
}
\end{equation}

\begin{lemma}
  If the inclusion 
$A_s^{p,*} (M,\mathfrak{U};V) \hookrightarrow A^{p,*} (M,\mathfrak{U};V)$ 
induces an isomorphism in cohomology, then the augmented complex 
$A_s^p (M;V) \hookrightarrow A_{sr}^{p,*} (M,\mathfrak{U};V)$ is exact.
\end{lemma}

\begin{proof}
 This is an immediate consequence of Diagram \ref{diaglecssmooth}
\end{proof}

\begin{proposition}
  If the inclusion 
$A_s^{p,*} (M,\mathfrak{U};V) \hookrightarrow A^{p,*} (M,\mathfrak{U};V)$ 
induces an isomorphism in cohomology, then the inclusions  
$j_{eq}^* : A_s^* (M;V)^G \hookrightarrow 
\tot A_{sr}^{*,*}(M,\mathfrak{U},V)^G$ and 
$A_s^* (M,\mathfrak{U};V)^G \hookrightarrow A_{sr}^* (M,\mathfrak{U};V)^G$
also induces an isomorphism in cohomology.
\end{proposition}

\begin{proof}
  This follows from the preceding Lemma and Corollaries 
\ref{augexthenjeqindisosmooth} and \ref{augextheninclindisosmooth}.
\end{proof}

The passage to the colimit over all open coverings of $M$ yields the
corresponding results for the complexes of cochains with continuous germs: 

\begin{proposition} \label{noneqextheneqexsg}
 If the augmented column complexes 
$A_s^p (M;V) \hookrightarrow A_{sg}^{p,*}(M;V)$ 
are exact, then the augmented sub column complexes
$A_s^p (M;V)^G \hookrightarrow A_{sg}^{p,*}(M;V)^G$  
of equivariant cochains are exact as well.
\end{proposition}

\begin{proof}
  The proof is similar to that of Proposition \ref{noneqextheneqex}.
\end{proof}

\begin{corollary} \label{augexthenjeqindisosg}
If the augmented column complexes 
$A_s^p (M;V) \hookrightarrow A_{sg}^{p,*} (M;V)$ 
are exact, then the inclusion 
$j_{eq}^* : A_s^* (M;V)^G \hookrightarrow 
\tot A_{sg}^{*,*}(M;V)^G$ induces an isomorphism in cohomology.
\end{corollary}

\begin{corollary} \label{augextheninclindisosg}
If the augmented column complexes 
 $A_s^p (M;V) \hookrightarrow A_{sg}^{p,*} (M;V)$ 
are exact, then the inclusion 
$A_s^* (M;V)^G \hookrightarrow A_{sg}^* (M;V)^G$
induces an isomorphism in cohomology.
\end{corollary}

\begin{remark} \label{remonlyginvcovsmooth}
Alternatively to taking the colimit over all open coverings $\mathfrak{U}$ of 
$M$ one may consider $G$-invariant open coverings only to obtains the same
results. 
(This was shown in Proposition \ref{natinclofeqccisisosmooth} and Lemmata 
\ref{natinclofeqdcisisosmooth}.)
\end{remark}

\begin{example}
  If $G=M$ is a Lie group which acts on itself by left translation and 
the augmented columns  
$A_s^p (M;V) \hookrightarrow A_{sg}^{p,*} (M;V) := 
\colim A^{p,*} (M,\mathfrak{U}_U;V)$ (where $U$ runs over all open identity
neighbourhoods in $G$) are exact, then the inclusion 
$A_s^* (M;V)^G \hookrightarrow A_{sg}^* (M;V)^G$
induces an isomorphism in cohomology.
\end{example}

The complex $A^{p,*} (M,\mathfrak{U};V)$ is isomorphic to the complex 
$A^* (\mathfrak{U}; C (M^{p+1},V))$. The colimit 
$A_{AS}^* (M; C (M^{p+1} ,V)):=\colim A^* (\mathfrak{U}; C (M^{p+1} ,V))$, 
where $\mathfrak{U}$ runs over all open coverings of
$M$ is the complex of Alexander-Spanier cochains on $M$. Therefore the colimit 
complex $\colim A^p (M; A^* (\mathfrak{U};V))$ is isomorphic to the cochain 
complex $A_{AS}^* (M;C (M^{p+1} ,V))$. 
A similar observation can be made for the cochain complex 
$A_s^{p,*} (M,\mathfrak{U};V)$ if the exponential law 
$C (M^{p+1} \times \mathfrak{U}[q],V) \cong C (M,C (\mathfrak{U}[q],V))$ 
holds for a cofinal set of open coverings $\mathfrak{U}$ of $M$. 
Passing to the colimit in Diagram \ref{morphexseqsmooth} yields the morphism

\begin{equation} \label{morphexseqsg}
\begin{array}{cccccccc}
 0 \longrightarrow & \ker (\res_{sg}^{p,*} ) & \longrightarrow & 
A_{sg}^{p,*} (M;V) & \longrightarrow & \colim A_s^{p,*} (M,\mathfrak{U};V) & 
\longrightarrow 0 \\ 
& \downarrow & & \downarrow & & \downarrow & \\
0 \longrightarrow & \ker (\res^{p,*} ) & \longrightarrow & A^{p,*} (M;V) 
& \longrightarrow & A_{AS}^* (M; C^{p+1} (M,V))
& \longrightarrow 0
\end{array}
\end{equation}
of short exact sequences of cochain complexes. The kernel $\ker (\res^{p,q})$
is the subspace of those $(p,q)$-cochains which are trivial on 
$M^{p+1} \times \mathfrak{U} [q]$ for some open covering $\mathfrak{U}$ of 
$M$. Since these $(p,q)$-cochains are continuous on 
$M^{p+1} \times \mathfrak{U}[q]$ we find that both kernels coincide. We 
abbreviate the complex $\ker (\res^{p,*} ) = \ker (\res_{sg}^{p,*} )$ by 
$K_{sg}^{p,*}$ and denote the cohomology groups of the complex 
$A_{sg}^{p,*} (M;V)$ by $H_{sg}^{p,*} (M;V)$. 
The morphism of short exact sequences of cochain complexes in Diagram 
\ref{morphexseqsg} gives rise to a morphism of long exact cohomology 
sequences:
\begin{equation} \label{diaglecssg}
\xymatrix@C-11pt@R-10pt{ 
\ar[r] & H^q (K_{sg}^{p,*} )) \ar[r] \ar@{=}[d] 
& H_{sg}^{p,q} (M,\mathfrak{U};V) \ar[r] \ar[d]  
& H^q (\colim A_s^{p,*} (M,\mathfrak{U};V) \ar[r] \ar[d]  
& H^{q+1}  (K_{sg}^{p,*} ) \ar[r] \ar@{=}[d] & {} \\ 
\ar[r]^\cong & H^q (K^{p,*} ) \ar[r]&  0  \ar[r] 
& H_{AS}^q (M; C^{p+1} (M,V)) \ar[r]^\cong &  H^{q+1}  (K^{p,*} )) \ar[r] & {}
}
\end{equation}

\begin{lemma}
  If the inclusion 
$\colim A_s^{p,*} (M,\mathfrak{U};V)\hookrightarrow A_{AS}^* (M;C^{p+1}(M,V))$
of cochain complexes induces an isomorphism in cohomology, then the augmented 
complex 
$A_s^p (M;V) \hookrightarrow A_{sg}^{p,*} (M;V)$ is exact.
\end{lemma}

\begin{proof}
 This is an immediate consequence of Diagram \ref{diaglecssg}
\end{proof}

\begin{proposition} \label{inclcolimacpsindisothenjeqisosmooth}
  If the inclusion 
$\colim A_s^{p,*} (M,\mathfrak{U};V)\hookrightarrow A_{AS}^* (M;C(M^{p+1},V))$ 
induces an isomorphism in cohomology, then  
$j_{eq}^* : A_s^* (M;V)^G \hookrightarrow \tot A_{sg}^{*,*}(M;V)^G$ and 
$A_s^* (M;V)^G \hookrightarrow A_{sg}^* (M;V)^G$
also induce an isomorphism in cohomology.
\end{proposition}

\begin{proof}
  This follows from the preceding Lemma and Corollaries 
\ref{augexthenjeqindisosg} and \ref{augextheninclindisosg}.
\end{proof}

As observed before (cf. Remark \ref{remonlyginvcovsmooth}) one may restrict 
oneself to the directed system of $G$-invariant open coverings only to achieve
the same result. Thus we observe:

\begin{corollary}
If $G=M$ is a Lie group which acts on itself by left translation and the 
inclusion 
$\colim A_s^{p,*} (M,\mathfrak{U};V)\hookrightarrow A_{AS}^* (M;C(M^{p+1},V))$ 
(where $U$ runs over all open identity
neighbourhoods in $G$)  induces an isomorphism in cohomology, then the 
inclusion $A_s^* (M;V)^G \hookrightarrow A_{sg}^* (M;V)^G$
induces an isomorphism in cohomology as well.
\end{corollary}

\begin{proof}
 It has been shown in \cite{vE62b} that the cohomology of the colimit cochain 
complex 
$\colim A^* ( \mathfrak{U} ;C(M^{p+1},V))$ is the Alexander-Spanier cohomology
of $M$.  
\end{proof}

\begin{lemma} \label{xsmoothlycontrthenacpsistriv}
If the manifold $M$ is contractible, then the cohomology 
of the complex $\colim A_s^{p,*} (M,\mathfrak{U};V)$ is trivial.
\end{lemma}

\begin{proof}
The reasoning is analogous to that for the Alexander-Spanier presheaf. 
The proof \cite[Theorem 2.5.2]{F10} carries over almost in verbatim. 
\end{proof}

\begin{proof}
If the manifold $M$ is contractible, then the Alexander-Spanier
cohomology $H_{AS} (M;C^{p+1}(M,V))$ is trivial and the cohomology of the
cochain complex $\colim A_s^{p,*} (M,\mathfrak{U};V)$ is trivial by Lemma 
\ref{xsmoothlycontrthenacpsistriv}. By Proposition
\ref{inclcolimacpsindisothenjeqisosmooth}
the inclusion $A_s^* (M;V)^G \hookrightarrow A_{sg}^* (M;V)^G$ then induces an
isomorphism in cohomology.
\end{proof}

\begin{corollary}
For smoothly contractible Lie groups $G$ the 
continuous group cohomology is isomorphic to the
cohomology of homogeneous group cochains with continuous germ at the diagonal.
\end{corollary}

\bibliographystyle{amsalpha}
\bibliography{ASpectralSequenceConnectingContinuousWithLocallyContinuousGroupCohomology}

\end{document}